\documentclass[12pt]{amsart}
\usepackage[margin=3cm]{geometry}

\usepackage{microtype} 

\usepackage{amsthm,amsmath,amssymb,amsfonts}
\usepackage[alphabetic]{amsrefs}
\usepackage[skip=4pt plus1pt, indent=10pt]{parskip} 
\usepackage{enumerate}
\usepackage{xcolor}
\usepackage{tikz}
\usepackage[colorlinks,bookmarks, anchorcolor=blue, citecolor=blue, linkcolor=blue,urlcolor=blue, bookmarksopenlevel=1]{hyperref}
\usepackage{tikz-cd}
\usepackage{appendix}
\usepackage{afterpage}
\usepackage{longtable}

\makeatletter
\newcommand{\addresseshere}{%
  \enddoc@text\let\enddoc@text\relax
}
\makeatother

\newtheorem{theorem}{Theorem}[section]
\newtheorem{proposition}[theorem]{Proposition}
\newtheorem{lemma}[theorem]{Lemma}
\newtheorem{corollary}[theorem]{Corollary}

\theoremstyle{definition}

\newtheorem{remark}[theorem]{Remark}
\newtheorem{example}[theorem]{Example}

\definecolor{darkgreen}{RGB}{0,90,0}

\DeclareMathOperator{\Id}{Id}

\begin{document}\baselineskip=15pt

\begin{center}
\title[Cones of NL divisors]{Cones of Noether--Lefschetz divisors and moduli spaces of hyperk\"ahler manifolds}

\author{Ignacio Barros}
\address{\parbox{0.9\textwidth}{
Department of Mathematics\\[1pt]
Universiteit Antwerpen\\[1pt]
Middelheimlaan 1, 2020 Antwerpen, Belgium
\vspace{1mm}}}
\email{{ignacio.barros@uantwerpen.be}}

\author{Pietro Beri}
\address{\parbox{0.9\textwidth}{
IECL--Site de Nancy\\[1pt]
Facult\'e des sciences et Technologies\\[1pt]
Campus, Boulevard des Aiguillettes\\[1pt]
54506 Vandœuvre-lès-Nancy, France
\vspace{1mm}}}
\email{{pietro.beri@univ-lorraine.fr}}

\author{Laure Flapan}
\address{\parbox{0.9\textwidth}{
Department of Mathematics\\[1pt]
Michigan State University\\[1pt]
619 Red Cedar Road, East Lansing, MI 48824, USA
\vspace{1mm}}}
\email{{flapanla@msu.edu}}

\author{Brandon Williams}
\address{\parbox{0.9\textwidth}{
University of Heidelberg \\[1pt]
Institute of Mathematics\\[1pt]
Im Neuenheimer Feld 205\\[1pt]
69120 Heidelberg, Germany
\vspace{1mm}}}
\email{{bwilliams@mathi.uni-heidelberg.de}}

\subjclass[2020]{14J15, 14C22, 14E08, 14J42.}
\keywords{K3 surfaces, moduli spaces, effective divisors, Heegner divisors.}
\thanks{I.B. was supported by the Research Foundation – Flanders (FWO) within the framework of the Odysseus program project number G0D9323N and by the Deutsche Forschungsgemeinschaft (DFG, German Research Foundation) -- SFB-TRR 358/1 2023 -- 491392403. 
P. B. was supported by the ERC Synergy Grant ERC-2020-SyG-854361-HyperK. 
L.F.~was supported by NSF grant DMS-2200800.}

\maketitle
\end{center}

\begin{abstract}
We give a general formula for generators of the NL cone on an orthogonal modular variety. This is the cone of effective divisors linearly equivalent to an effective linear combination of irreducible components of Noether-Lefschetz divisors. We apply this to describe, in terms of minimal generators, the NL cone of various moduli spaces of geometric origin such as those of polarized K3 surfaces, cubic fourfolds, and hyperk\"ahler manifolds. Additionally, we establish uniruledness for many moduli spaces of primitively polarized hyperk\"ahler manifolds of ${\rm{OG6}}$ and ${\rm{Kum}}_n$-type. Finally, in analogy with the case of K3 surfaces of degree $2$, we show that any family of polarized ${\rm{Kum}}_2$-type hyperk\"ahler manifolds with divisibility $2$ and polarization degree $2$ over a projective base is isotrivial. 
\end{abstract}

\setcounter{tocdepth}{1}

\section{Introduction}
Two invariants governing the birational geometry of a variety $X$ are its Kodaira dimension and its cone of pseudo-effective divisors $\overline{{\rm Eff}}(X)$. The cone $\overline{{\rm Eff}}(X)$ is defined as the closure in ${\rm{Pic}}_{\mathbb{R}}\left(X\right)$ of the cone of effective $\mathbb{R}$-divisors on $X$. This cone often admits a decomposition into chambers each representing a birational model of the variety. Further, extremal rays of $\overline{{\rm Eff}}(X)$ often arise as the divisorial exceptional locus of birational contractions of $X$.  In general, it can be quite difficult to determine when $\overline{{\rm Eff}}(X)$ is finitely generated, let alone describe it explicitly. 

In the case of the moduli space $\mathcal{F}_{2d}$ of quasi-polarized K3 surfaces of degree $2d$,  the most natural source of effective divisors is Noether--Lefschetz divisors. A very general point $(S,H)\in \mathcal{F}_{2d}$ has Picard group ${\rm Pic}(S)=\mathbb{Z}H$ and so the locus in $\mathcal{F}_{2d}$ where $\rho(S)\geq2$ is a countable union of divisors,  called Noether--Lefschetz divisors (or NL divisors). Concretely, a Noether--Lefschetz divisor $\mathcal{D}_{h,a}$ on $\mathcal{F}_{2d}$ is the reduced divisor obtained by taking the closure of the locus of points $(S,H)\in \mathcal{F}_{2d}$ for which there exists a class $\beta\in {\rm Pic}(S)$, not proportional to $H$, with $\beta^2=2h-2$ and $\beta.H=a$. Maulik--Pandharipande conjectured \cite{MP13}*{Conjecture 3} that the rational Picard group ${\rm Pic}_\mathbb{Q}(\mathcal{F}_{2d})$ is generated by Noether--Lefschetz divisors $\mathcal{D}_{h,a}$. 

The moduli space $\mathcal{F}_{2d}$ arises as a quotient $\mathcal{D}/\widetilde{{\rm{O}}}^+\left(\Lambda_{2d}\right)$, where $\mathcal{D}$ is the Type IV symmetric domain associated with ${\rm O}(2,n)$, the lattice $\Lambda_{2d}$ is the  even lattice of signature $(2,n)$ given by 
\[
\Lambda_{2d}=U^{\oplus 2}\oplus E_8(-1)^{\oplus 2}\oplus \mathbb{Z}\ell,\;\;\hbox{with}\;\;\langle\ell,\ell\rangle=-2d,
\]
and $\widetilde{{\rm{O}}}^+\left(\Lambda_{2d}\right)$ is the group of orientation-preserving isomorphisms of $\Lambda_{2d}$ which act trivially on the discriminant group $D(\Lambda_{2d})=\Lambda_{2d}^\vee\big/\Lambda_{2d}$.
For an arbitrary orthogonal modular variety $\mathcal{D}/\Gamma$, where
 $\Gamma$ is a finite index subgroup of $\widetilde{{\rm{O}}}^+\left(\Lambda\right)$ and $\Lambda$ an even lattice of signature $(2,n)$, Noether--Lefschetz divisors generalize to Heegner divisors, which are images of hyperplane arrangements in $\mathcal{D}$ under the modular projection $\pi:\mathcal{D}\longrightarrow \mathcal{D}\big/\Gamma$.

Bergeron--Li--Millson--Moeglin \cite{BLMM17}  and Bruinier--Zuffetti \cite{BZ24} proved a generalization of Maulik--Pandharipande's conjecture, showing that when $n\geq 3$ and $\Lambda$ splits off two copies of the hyperbolic plane, the Picard group with rational coefficients ${\rm Pic}_\mathbb{Q}(X)$ of any orthogonal modular variety $X=\mathcal{D}/\widetilde{{\rm{O}}}^+\left(\Lambda\right)$ is generated by Heegner divisors. The rank of ${\rm Pic}_\mathbb{Q}(\mathcal{D}/\Gamma)$ was computed by Bruinier in \cite{Bru02b}. 

For any orthogonal modular variety $X=\mathcal{D}/\Gamma$ as above, the NL cone ${\rm{Eff}}^{NL}\left(X\right)\subset {\rm{Pic}}_{\mathbb{Q}}\left(X\right)$ is the convex cone of effective $\mathbb{Q}$-linear combinations of irreducible components of Heegner divisors (known as primitive Heegner divisors) on $X$. The NL cone contains the subcone ${\rm{Eff}}^{H}\left(\mathcal{D}/\Gamma\right)$ generated by the (non-primitive) Heegner divisors on $\mathcal{D}/\Gamma$. After tensoring with $\mathbb{R}$, this NL cone ${\rm{Eff}}^{NL}\left(X\right)$ forms a natural subcone of the cone of pseudo-effective divisors $\overline{{\rm Eff}}(X)$. 

The study of NL cones was initiated in \cite{Pet15} in the case $X=\mathcal{F}_{2d}$, where the following three questions are raised \cite{Pet15}*{Section 4.5}:
\begin{enumerate}
\item \label{petQ1}Is ${\rm{Eff}}^{NL}\left(\mathcal{F}_{2d}\right)$ finitely-generated (polyhedral)?
\item \label{petQ2}Can we compute generators for ${\rm{Eff}}^{NL}\left(\mathcal{F}_{2d}\right)$?
\item \label{petQ3}Is there equality ${\rm{Eff}}^{NL}\left(\mathcal{F}_{2d}\right)=\overline{{\rm Eff}}(\mathcal{F}_{2d})$?
\end{enumerate}
Bruinier--M\"oller \cite{BM19} answered the first question affirmatively, showing that for any orthogonal modular variety $X=\mathcal{D}\big/\widetilde{\rm{O}}^+\left(\Lambda\right)$ with $n\ge 3$ splitting off two copies of the hyperbolic plane, the cone ${\rm{Eff}}^{NL}\left(X\right)$ is always polyhedral.

In this paper, we tackle Question \eqref{petQ2} for $X=\mathcal{D}\big/\widetilde{\rm{O}}^+\left(\Lambda\right)$ under the same assumptions. We consider the $\mathbb{Q}$-vector space $S_{k,\Lambda}$ of vector-valued cusp forms of weight $k=1+\frac{n}{2}$ with respect to the Weil representation \cite{Bor98} and the coefficient extraction functionals in $S_{k,\Lambda}^\vee$
 \[
 c_{m, \mu} : S_{k,\Lambda} \longrightarrow \mathbb{Q}, \quad \sum a_{m, \mu} q^m \mathfrak{e}_{\mu} \mapsto a_{m, \mu},
 \] 
 where $\{\mathfrak{e}_{\mu}\mid \mu\in D(\Lambda)\}$ are the standard generators of the
 group algebra $\mathbb{C}\left[D(\Lambda)\right]$.
 Let $b\ge \lceil k/12\rceil$ be an integer such that the set of $c_{m, \mu}$ with $0<m\le b$ and $\mu \in D(\Lambda)$ generates $S_{k,\Lambda}^\vee$. Then, we consider the weakly holomorphic modular form
\begin{equation}\label{eq: weakly hol}
\Delta^{-b}\cdot E_{(2-k)+12b, \Lambda(-1)}=\sum_{\substack{(m,\mu)\\ -b\leq m}}\alpha_{m,\mu}q^m\mathfrak{e}_\mu,
\end{equation}
where $\Delta(\tau)$ is the scalar-valued discriminant modular form and $E_{(2-k)+12b, \Lambda(-1)}$ is the Eisenstein series of weight $(2-k)+12b$ associated to $\Lambda(-1)$ (see Equation \eqref{sec2:eq:E}).

 Our first result is the following, with the explicit bounds in Theorems \ref{thm: nonprimitive cone generators} and \ref{thm: NL cone generators}.

\begin{theorem}
\label{int:thm:1}
Let $\Lambda$ be an even lattice of signature $(2,n)$ with $n\ge 4$ splitting off two copies of the hyperbolic plane and $X=\mathcal{D}\big/\widetilde{\rm{O}}^+(\Lambda)$ its modular variety. Fixing $b$ as above, there are explicit bounds $\Xi$ and $\Omega$, depending on $k$, the discriminant of $\Lambda$, and the $\alpha_{m, \mu}$ with $-b\le m\le 0$ in \eqref{eq: weakly hol},  such that
\begin{enumerate}
    \item The cone ${\rm{Eff}}^H\left(X\right)$ is generated by all $H_{-m, \mu}$ with $0\le m\le \Xi$. 
    \item The cone ${\rm{Eff}}^{NL}\left(X\right)$ is generated by all $P_{-\Delta, \delta}$ with $0\le \Delta \le \Omega$.
\end{enumerate}
\end{theorem}

Theorem \ref{int:thm:1} together with its implementation in a Sage package \cite{heegner_cones} enables the description of ${\rm{Eff}}^{NL}\left(X\right)$ in terms of generators for any such $\Lambda$ (see Section \ref{sec: explicit comp} below). 

We remark that the bound $\Xi$ of Theorem \ref{thm: nonprimitive cone generators} in fact allows for the description of the cone in $\left({\rm Mod}^\circ_{k, \Lambda}\right)^\vee$ generated by coefficient extraction functionals $c_{m, \mu}\colon {\rm Mod}^\circ_{k, \Lambda}\rightarrow \mathbb{Q}$ under the weaker assumption that $\Lambda$ splits off only one copy of the hyperbolic plane. Here ${\rm Mod}^\circ_{k, \Lambda}=\mathbb{Q}E_{k, \Lambda} \oplus S_{k, \Lambda}$ is the $\mathbb{Q}$-vector space of almost cusp forms. The assumption that $\Lambda$ splits off two copies of the hyperbolic plane is needed to convert the above result into a description of ${\rm{Eff}}^{NL}\left(X\right)$ via the results of \cites{BLMM17, BZ24}.

The proof of Theorem \ref{int:thm:1} relies on the relationship between Heegner divisors on $X$ and vector-valued modular forms with respect to the Weil representation for $\Lambda$. In \cite{BM19} the polyhedrality of the NL cone is established by showing that the Hodge class $\lambda$ lies in the interior of the NL cone, and the rays generated by primitive Heegner divisors  converge to $\lambda\mathbb{Q}_{\geq0}$. Establishing a concrete list of generators of ${\rm{Eff}}^{NL}\left(X\right)$ amounts to making the convergence rate explicit which translates into bounding explicitly the growth of the coefficients of the relevant vector-valued modular forms (see Section \ref{sec: NL cone comp}). For vector-valued cusp forms of half-integer weight, despite the considerable literature on bounds for the growth of Fourier coefficients, we are unaware of a general bound with explicit constants. Using Poincar\'e series and Kloosterman sums we derive weak, yet explicit, bounds that suffice for our purposes. 

\subsection{Applications to moduli}\label{sec: explicit comp}

We then focus on cases where the quotient $X=\mathcal{D}\big/\widetilde{\rm{O}}^+\left(\Lambda\right)$ arises as a partial compactification of a coarse moduli space of polarized K3 surfaces, hyperk\"ahler manifolds, or cubic fourfolds. We give explicit formulas for ${\rm{Eff}}^{NL}\left(X\right)$ in terms of generating rays for low-degree polarizations: see Table \ref{sec6:table:NLK3} for the case of (quasi)-polarized K3 surfaces and Tables \ref{sec6:table:NLK32_sp} and \ref{sec6:table:NLK32_nsp} for the case of hyperk\"ahler fourfolds of ${\rm K3}^{[2]}$-type. We remark that the orthogonal modular variety $X=\mathcal{D}\big/\widetilde{\rm{O}}^+\left(\Lambda\right)$ partially compactifying the moduli space of smooth cubic fourfolds is the same as that partially compactifying the moduli space of polarized hyperk\"ahler manifolds of ${\rm K3}^{[2]}$-type with polarization of divisibility $2$ and degree $6$ and thus the description of ${\rm{Eff}}^{NL}\left(X\right)$ for cubic fourfolds is already contained in Table \ref{sec6:table:NLK32_nsp}. In the case of $\mathcal{F}_{2d}$, the calculations in Table \ref{sec6:table:NLK3} confirm (aside from one additional necessary generator in the case $d=13$) the predictions in \cite{Pet15} who computed, for $d\leq 18$, the cone generated by the set of $8d$ primitive Heegner divisors $P_{\Delta,\delta}$, for $\delta\in D(\Lambda), \Delta\in Q(\delta)+s$ with $s=0,1,2,3$, and conjectured that this cone coincides with the one generated by all of them, that is, with ${\rm{Eff}}^{NL}\left(\mathcal{F}_{2d}\right)$. 

In some cases, one can use the position of the canonical class $K_X$ with respect to the NL cone to show that $X$ has negative Kodaira dimension. This occurs when $K_X$ lies on the opposite side from the NL cone of the hyperplane in ${\rm Pic}_\mathbb{Q}(X)$ of divisors with degree $0$ with respect to the Hodge class $\lambda$. We formalize this condition numerically in terms of the Eisenstein series (see the more general Proposition \ref{prop: uniruled with branch}) in order to give the following criterion for uniruledness.  
\begin{proposition}\label{prop: general uniruledness} Let $\Lambda$ be an even lattice of signature $(2,n)$ with $n\geq3$ splitting off two copies of $U$ and let  $E_{\frac{n+2}{2},\Lambda}$ be its Eisenstein series.  If \\\[nc_{0,0}\left(E_{\frac{n+2}{2},\Lambda}\right)+\frac{1}{4}c_{1,0}\left(E_{\frac{n+2}{2},\Lambda}\right)<0,\] then the orthogonal modular variety $X=\mathcal{D}\big/\widetilde{{\rm{O}}}^+(\Lambda)$ is uniruled. 
\end{proposition}

\subsection{Uniruledness results}
Mukai in a celebrated series of papers \cites{Muk88,Muk92,Muk06,Muk10,Muk16} constructed unirational parameterizations of $\mathcal{F}_{2d}$ for low degrees. This has been recently improved by Farkas--Verra in \cites{FV18, FV21}. The first examples of higher dimensional projective hyperk\"{a}hler varieties were exhibited in \cite{Bea83}: they are generalized Kummer varieties and Hilbert schemes of points on K3 surfaces. Polarized varieties of these two types deform in $4$ and $20$ dimensional moduli spaces respectively. 

The problem of exhibiting a projective realization of a generic such object is intimately related with rationality properties of the corresponding moduli space.  Although constructions of ${\rm Kum}_n$ and ${\rm{K3}}^{[n]}$-type hyperk\"{a}hler varieties were exhibited at the same time, unirational parameterizations are available only for some moduli spaces of hyperk\"{a}hler varieties of ${\rm{K3}}^{[n]}$-type, see \cites{BD85, OG06, IR01, IR07, DV10, BLM+21}. Constructing unirational parameterizations in low degree for moduli spaces of other types of hyperk\"{a}hler varieties (e.g. generalized Kummer, ${\rm{OG}}6$, and ${\rm{OG}}10$ types) presents a challenge where, as far as we know, not a single explicit construction is known. Here we consider the simpler problem of establishing uniruledness and focus on the generalized Kummer and ${\rm{OG}}6$-type cases (our methods do not yield new results in the case of ${\rm{OG}}10$-type).

In Section \ref{sec: uniruled}, we consider the moduli spaces $\mathcal{M}_{{\rm{OG6}},2d}^{\gamma}$ and $\mathcal{M}_{{\rm{Kum}}_n,2d}^\gamma$, which are the period domain partial compactifications of the moduli spaces $\left(\mathcal{M}_{{\rm{OG6}},2d}^{\gamma}\right)^\circ$ and $\left(\mathcal{M}_{{\rm{Kum}}_n,2d}^\gamma\right)^\circ$ parameterizing primitively polarized hyperk\"{a}hler sixfolds of OG6-type respectively $2n$-folds of ${\rm{Kum}}_n$-type with a primitive polarization of degree $2d$ and divisibility $\gamma$. We remark that the moduli space $\mathcal{M}_{{\rm{OG6}},2d}^\gamma$ is always irreducible and in the case $\gamma=2$ it is non-empty only when $d\equiv -1,-2 \mod 4$. Similarly, setting $d=1$ and $\gamma\in\{1,2\}$, the moduli space $\mathcal{M}_{{\rm{Kum}}_n,2}^\gamma$ is irreducible and in the case $\gamma=2$ its nonempty only when $n\equiv 2\mod 4$.

Theorems \ref{sec7:them:OG6uniruled} and \ref{sec7:thm:kum_uniruled} establish the following uniruledness results:
\begin{theorem}
\label{int:thm:HKs}
The moduli space $\mathcal{M}_{{\rm{OG6}},2d}^\gamma$ is uniruled in the following cases
\begin{itemize}

\item[(i)] when $\gamma=1$ for $d\leq12$,
\item[(ii)] when $\gamma=2$ for $t\leq10$ and $t=12$ with $d=4t-1$,
\item[(iii)] when $\gamma=2$ for $t\leq9$ and $t=11, 13$ with $d=4t-2$.
\end{itemize}
The moduli spaces $\mathcal{M}_{{\rm{Kum}}_n,2}^1$ and $\mathcal{M}_{{\rm{Kum}}_n,2}^2$ are uniruled in the following cases:
\begin{itemize}
\item[(i)] when $\gamma=1$ for $n\leq 15$ and $n=17,20$,
\item[(ii)] when $\gamma=2$ for $t\leq 11$ and $t=13,15,17,19$, where $n=4t-2$.
\end{itemize}
\end{theorem}

An immediate consequence of the work of H. Wang and the fourth author \cite{WW21}*{Theorem 5.4} together with Lemmas \ref{sec7:lemma:monOG6} and \ref{sec7:lemma:MonKum2_2}, appearing here, is the rationality of $\mathcal{M}_{{\rm{Kum}}_2,2}^2$ and unirationality of $\mathcal{M}_{{\rm{OG6}},6}^2$ and $\mathcal{M}_{{\rm{OG6}},2}^1$. In the case of the rational moduli space $\left(\mathcal{M}_{{\rm{Kum}}_2,2}^2\right)^\circ$ we moreover establish the following.

\begin{theorem}\label{thm: quasi-affine}
The moduli space $\left(\mathcal{M}_{{\rm{Kum}}_2,2}^2\right)^\circ$ parameterizing polarized hyperk\"{a}hler fourfolds with polarization of degree $2$ and divisibility $2$ is quasi-affine.  
\end{theorem}

When considering projective hyperk\"{a}hler varieties, it is natural to study families of such. In particular, if $\mathcal{X}\longrightarrow B$ is a non isotrivial family of polarized hyperk\"{a}hler varieties of certain type, can one say something about $B$? This was first treated in \cite{BKPS98}*{Theorem 1.3} for K3 surfaces of degree two where it is established that $B$ cannot be projective. This problem was further studied in \cite{DM22}. An immediate consequence of Theorem \ref{thm: quasi-affine} is:

\begin{corollary}\label{cor: isotrivial}
Any family $f:\mathcal{X}\longrightarrow B$ over a projective base $B$ of polarized hyperk\"{a}hler fourfolds of ${\rm{Kum}}_2$-type with polarization of degree $2$ and divisibility $2$ is isotrivial. 
\end{corollary}

\subsection*{Acknowledgements} This paper benefited from helpful discussions and correspondence with the
following people who we gratefully acknowledge: Daniele Agostini, Emma Brakkee, Jan Hendrik Bruinier, Yagna Dutta, Gabi Farkas, Paul Kiefer, Giovanni Mongardi, Gregory Sankaran, Preston Wake, and Riccardo Zuffetti. We would also like to thank the anonymous referee for helpful comments and corrections.

\section{Preliminaries}

Let $\Lambda$ be an even lattice of signature $
(2,n)$ with bilinear form given by $\langle\cdot,\cdot\rangle$. The bilinear form extends to $\Lambda_{\mathbb{C}}$ and we let $\mathcal{D}$ denote one of the two components of
\[
\left\{[Z]\in \mathbb{P}\left(\Lambda_{\mathbb{C}}\right)\left|\langle Z,Z\rangle=0, \langle Z,\overline{Z}\rangle>0 \right.\right\}.
\]
Further, we denote by $\Gamma$ a finite index subgroup of the group ${\rm{O}}^+\left(\Lambda\right)$ of automorphisms of $\Lambda$ fixing the component $\mathcal{D}$. The quotient of $\mathcal{D}$ by $\Gamma$ is called an {\textit{orthogonal modular variety}}. It is a quasi-projective variety \cite{BB66} which for some choices of lattice $\Lambda$ and arithmetic groups $\Gamma$ gives a partial compactification of a coarse moduli space of polarized varieties. The first case of interest in this paper is when
\[
\Lambda_{2d}=U^{\oplus 2}\oplus E_8(-1)^{\oplus 2}\oplus \mathbb{Z}\ell,\;\;\hbox{with}\;\;\langle\ell,\ell\rangle=-2d
\]
and the arithmetic group $\Gamma=\widetilde{{\rm{O}}}^+\left(\Lambda\right)$ is the group of orientation preserving isomorphisms of $\Lambda$ acting trivially on the discriminant group $D(\Lambda)=\Lambda^\vee\big/\Lambda$. The quotient
\[
\mathcal{F}_{2d}=\mathcal{D}\big/\widetilde{{\rm{O}}}^+\left(\Lambda_{2d}\right)
\]
is the moduli space for quasi-polarized K3 surfaces $(S,H)$, i.e., where $H$ is primitive, big, and nef, of degree $H^2=2d$. 

As mentioned in the introduction, a very general point $(S,H)\in \mathcal{F}_{2d}$ has Picard group ${\rm{Pic}}(S)=\mathbb{Z}H$, and a large source of geometric divisors comes from imposing the condition that the Picard rank jumps. These are {\textit{Noether--Lefschetz divisors}}. There are different characterizations of these divisors: by keeping track of a rank two lattice embedding $L\hookrightarrow {\rm{Pic}}(S)$, by imposing the existence of an extra class $\beta\in {\rm{Pic}}(S)$ with fixed intersections $(\beta^2, \beta\cdot H)=(2h-2,a)$, and by looking at images of hyperplanes in $\mathcal{D}$ via the quotient map 
\[
\pi_{2d}:\mathcal{D}\longrightarrow \mathcal{F}_{2d}.
\]
These are all equivalent approaches (see \cite{MP13}*{Section 1 and Lemma 3}). In what follows, we focus on the third approach. 

\subsection{Heegner and NL divisors}

We assume $\Gamma\subset \widetilde{\rm{O}}^+(\Lambda)$. Let $Q(x)=\frac{\langle x,x\rangle}{2}$ be the corresponding quadratic form. For fixed $v\in \Lambda^{\vee}\subset\Lambda_{\mathbb{Q}}$, we set 
\[
D_v=v^\perp\cap \mathcal{D}=\left\{[Z]\in \mathcal{D}\left|\langle Z,v\rangle=0\right.\right\}.
\]
Let $\mu+\Lambda\in \Lambda^\vee\big/\Lambda$ and $m\in Q(\mu)+\mathbb{Z}$ negative. Then the cycle 
\begin{equation}
\label{sec2:eq:Heegner}
\sum_{\substack{v\in \mu+\Lambda\\Q(v)=m}}D_v
\end{equation}
is $\Gamma$-invariant and descends to a $\mathbb{Q}$-Cartier divisor $H_{m,\mu}$ called a {\textit{Heegner divisor}}. In general, $H_{m,\mu}$ is neither reduced, nor irreducible. The existence of two vectors $v, v'\in \Lambda^{\vee}$ with the same square and discriminant class for which $D_v=D_{v'}$ is a source for non-reduced components of $H_{m, \mu}$. Similarly, several $\Gamma$-orbits of elements in $\Lambda^\vee$ with the same square and discriminant class give rise to several (possibly non-reduced) components.

Under the given assumption that $\Gamma\subset \widetilde{\rm{O}}^+\left(\Lambda\right)$, all the components of $H_{m,\mu}$ have multiplicity two if $\mu=-\mu$ in $\Lambda^\vee\big/\Lambda$ and all have multiplicity one otherwise. Further, the line bundle $\mathcal{O}(-1)$ on $\mathcal{D}\subset \mathbb{P}\left(\Lambda_{\mathbb{C}}\right)$ admits a natural $\Gamma$-action and descends to a $\mathbb{Q}$-line bundle $\lambda$ called the {\textit{Hodge bundle}}. One declares $H_{0,0}=-\lambda$. 

In the K3 case $\mathcal{F}_{2d}=\mathcal{D}\big/\widetilde{{\rm{O}}}^+\left(\Lambda_{2d}\right)$, Noether-Lefschetz divisors are often described as the reduced divisor obtained by taking the closure of the locus 
\[
\mathcal{D}_{h,a}\subset\mathcal{F}_{2d}
\]
of points $(S,H)$ for which there exists a class $\beta\in {\rm{Pic}}(S)$ not proportional to $H$ with $\beta^2=2h-2$ and $\beta\cdot H=a$. In this case \cite{MP13}*{Lemma 3}, if $d$ does not divide $a$:
\[\mathcal{D}_{h,a}=H_{-m,\mu}\;\;\hbox{ with }m=\frac{a^2}{4d}-(h-1),\;\;\hbox{ and }\mu=a\ell_*.
\]
Here $\ell_*=\frac{\ell}{2d}\in D(\Lambda_{2d})$ is the standard generator. If $d$ divides $a$, then $\mathcal{D}_{h,a}=\frac{1}{2}H_{m,\mu}$. One denotes by ${\rm{Pic}}_{\mathbb{Q}}^{H}\left(\mathcal{F}_{2d}\right)$ the subspace generated by all NL divisors $\mathcal{D}_{h,a}$, or equivalently, Heegner divisors $H_{m,\mu}$. Maulik--Pandharipande conjectured \cite{MP13}*{Conjecture 3} the equality 
\[
{\rm{Pic}}_{\mathbb{Q}}^{H}\left(\mathcal{F}_{2d}\right)={\rm{Pic}}_{\mathbb{Q}}\left(\mathcal{F}_{2d}\right).
\]
This is now a theorem:
\begin{theorem}\cite[Theorem 1.5]{BLMM17}, \cite[Remark 3.13, Corollary 3.18]{BZ24}
\label{thm:sec2:BLMM17}
Let $\Lambda$ be an even lattice of signature $(2,n)$ with $n\geq3$ splitting off two copies of the hyperbolic plane. Then the rational Picard group of $\mathcal{D}\big/\widetilde{{\rm{O}}}^+\left(\Lambda\right)$ is generated by Heegner divisors:
\[
{\rm{Pic}}_{\mathbb{Q}}^H\left(\mathcal{D}\big/\widetilde{{\rm{O}}}^+\left(\Lambda\right)\right)={\rm{Pic}}_{\mathbb{Q}}\left(\mathcal{D}\big/\widetilde{{\rm{O}}}^+\left(\Lambda\right)\right).
\]
\end{theorem}
Note that the above theorem in particular implies that irreducible components of $H_{m,\mu}$ must be linear combinations of other Heegner divisors. This relation is explicit and follows from Eichler's criterion \cite{GHS09}*{Proposition 3.3}, \cite{Son23}*{Proposition 2.15}: {\textit{if $\Lambda$ splits off two copies of the hyperbolic lattice $U$, then the $\widetilde{\rm{SO}}^{+}\left(\Lambda\right)$-orbit of a primitive element $v\in \Lambda^\vee$ is determined by $Q(v)=m$ and $v+\Lambda\in \Lambda^\vee\big/\Lambda$}}. This leads to the following definition (see \cites{Pet15, BM19}). The {\textit{primitive Heegner divisor}} $P_{\Delta,\delta}$ is the image via the $\Gamma$-quotient map $\pi:\mathcal{D}\longrightarrow \mathcal{D}\big/\Gamma$ of the cycle
\begin{equation}
\label{eq:sec2:P_{m,mu}}
\sum_{\substack{v\in \delta+\Lambda\;\;primitive\\Q(v)=\Delta}}D_v.
\end{equation}
Under the assumption that $\Lambda$ splits off two copies of $U$ and $\Gamma=\widetilde{\rm{O}}^+\left(\Lambda\right)$, the divisor $P_{\Delta,\delta}$ is always irreducible. Moreover, in this situation, the divisor $P_{\Delta,\delta}$ is reduced when $\delta\neq -\delta$ in $D(\Lambda)$ and has multiplicity two otherwise. The relation between Heegner and primitive Heegner divisors \cite{BM19}*{Equations (17) and (18)} is:
\begin{equation}
\label{sec2:eq:H-P}
H_{m,\mu}=\sum_{\substack{r\in\mathbb{Z}_{>0}\\r^2\mid m}}\sum_{\substack{\delta\in D(\Lambda)\\ r\delta=\mu}}P_{\frac{m}{r^2}, \delta}\;\;\hbox{and}\;\;P_{\Delta,\delta}=\sum_{\substack{r\in\mathbb{Z}_{>0}\\r^2\mid \Delta}}\mu(r)\sum_{\substack{\sigma\in D(\Lambda)\\ r\sigma=\delta}}H_{\frac{\Delta}{r^2},\sigma},
\end{equation}
where $r^2\mid m$ denotes that there is a class $\delta \in D(\Lambda)$ such that $m/r^2\in Q(\delta)+\mathbb{Z}$ and 
the $\mu(\cdot)$ in the second equation stands for the M\"{o}bius function.

As stated in the introduction, our main object of study is the {\textit{NL cone}} ${\rm{Eff}}^{NL}\left(\mathcal{D}\big/\Gamma\right)$ generated by primitive Heegner divisors, or equivalently irreducible components of Noether-Lefschetz divisors.

\subsection{Rational Picard group of orthogonal modular varieties}\label{ss:Pic}

A recently established key feature of our setting is that the $\mathbb{Q}$-vector space ${\rm{Pic}}_{\mathbb{Q}}\left(\mathcal{D}\big/\widetilde{\rm{O}}^+\left(\Lambda\right)\right)$ can be seen as a space of vector-valued modular forms. This is what we explain now. 

Let $\Lambda$ be an even lattice of signature $(2,n)$ with quadratic form $Q$. The discriminant group $D(\Lambda)=\Lambda^\vee\big/\Lambda$ is a finite abelian group endowed with an induced $\mathbb{Q}\big/\mathbb{Z}$-valued quadratic form. The group algebra $\mathbb{C}\left[D(\Lambda)\right]$ is finitely generated and we denote the standard generators by $\{\mathfrak{e}_{\mu}\mid \mu\in D(\Lambda)\}$. The {\textit{metaplectic group}} ${\rm{Mp}}_2(\mathbb{Z})$ is a double cover of ${\rm{SL}}_2(\mathbb{Z})$ defined as the group of pairs $(A,\phi(\tau))$ where $A=\left(\begin{array}{cc}a&b\\c&d\end{array}\right)\in {\rm{SL}}_2\left(\mathbb{Z}\right)$, and $\phi(\tau)$ is a choice of a square root of the function $c\tau+d$ on the upper half plane $\mathbb{H}$. The product in ${\rm{Mp}}_2(\mathbb{Z})$ is given by $(A_1,\phi_1(\tau))\cdot(A_2,\phi_2(\tau))=(A_1A_2, \phi_1(A_2\tau)\phi_2(\tau))$. There is a canonical representation of ${\rm{Mp}}_2(\mathbb{Z})$ attached to $\Lambda$ called the {\textit{Weil representation}} $\rho_\Lambda:{\rm{Mp}}_2(\mathbb{Z})\longrightarrow {\rm{GL}}\left(\mathbb{C}\left[D(\Lambda)\right]\right)$. See \cite{Bor98}*{Section 4} for a concrete description in terms of the standard generators of ${\rm{Mp}}_2(\mathbb{Z})$. Let $k\in\frac{1}{2}\mathbb{Z}$. A holomorphic function 
\[
f:\mathbb{H}\longrightarrow \mathbb{C}\left[D(\Lambda)\right]
\]
is called a {\textit{modular form of weight $k$ and type $\rho_{\Lambda}$}} if for all $g=(A,\phi)\in {\rm{Mp}}_2\left(\mathbb{Z}\right)$ and $\tau\in \mathbb{H}$
\[
f(A\tau)=\phi(\tau)^{2k}\rho_{\Lambda}(g)\cdot f(\tau) 
\]
and $f$ is holomorphic at the cusp at $\infty$. Modular forms of weight $k$ and type $\rho_{\Lambda}$ form a finite-dimensional $\mathbb{C}$-vector space denoted ${\rm{Mod}}_{k,\Lambda}$. Such a modular form $f$ admits a Fourier expansion centered at the cusp at infinity of the form
\[
f=\sum_{\mu\in D(\Lambda)}\sum_{m\in \frac{1}{N}\mathbb{Z}_{\geq0}}a_{m,\mu}q^m\mathfrak{e}_{\mu},
\]
where as usual $q=e^{2\pi i \tau}$. Here $N$ is the {\textit{level}} of $\Lambda$, that is, the smallest positive integer such that $N\cdot Q$ is integral on $\Lambda^\vee$. Further, from \cite{Bor99}*{Lemma 4.2} and \cite{McG03}*{Theorem 5.6}, one can find a basis for ${\rm{Mod}}_{k,\Lambda}$ where all Fourier coefficients are rational numbers.

The modular form $f$ is called a {\textit{cusp form}} if $a_{0,\mu}=0$ for all isotropic elements $\mu\in D(\Lambda)$, i.e, the function $\sum_{m}a_{m,\mu}q^m$ vanishes at the cusp of $\mathbb{H}$. The function $f$ is called an {\textit{almost cusp form}} if $a_{m,\mu}=0$ for all isotropic elements $\mu$ except possibly $0\in D(\Lambda)$ (see for instance \cite{Pet15}*{Section 3.3}). Letting ${\rm{S}}_{k,\Lambda}$ denote the space of cusp forms and ${\rm{Mod}}_{k,\Lambda}^\circ$ the space of almost cusp forms, we have subspace inclusions 
\[
{\rm{S}}_{k,\Lambda}\subset {\rm{Mod}}_{k,\Lambda}^\circ\subset {\rm{Mod}}_{k,\Lambda}.
\]
Let $\widetilde{\Gamma}_{\infty}$ be the stabilizer in ${\rm{Mp}}_2(\mathbb{Z})$ of the cusp at infinity. If $k>2$ is a half integer such that $2k\equiv 2-n\mod 4$, then the Eisenstein series
\begin{equation}
\label{sec2:eq:E}
E_{k,\Lambda}(\tau)=\sum_{(A,\phi)\in \widetilde{\Gamma}_{\infty}\setminus{\rm{Mp}}_2(\mathbb{Z})}\phi(\tau)^{2k}\cdot \rho_{\Lambda}(A,\phi)^{-1}\mathfrak{e}_0=\sum_{m,\mu}e_{m,\mu}q^m\mathfrak{e}_\mu 
\end{equation}
is in ${\rm{Mod}}_{k,\Lambda}$. The coefficients $e_{m,\mu}$ are rational numbers that were computed in \cite{BK01}. As $\mathbb{Q}$-vector spaces one has
\[
{\rm{Mod}}_{k,\Lambda}^\circ=\mathbb{Q}E_{k,\Lambda}\oplus S_{k,\Lambda}.
\]
Following the notation in \cites{Pet15,BM19}, consider the coefficient extraction functionals 
\[
\begin{array}{rcl}
c_{m,\mu}:{\rm{Mod}}_{k,\Lambda}^\circ&\longrightarrow&\mathbb{Q}\\
f&\mapsto&c_{m,\mu}(f).
\end{array}
\]
where $c_{m,\mu}(f)$ is the $(m,\mu)$-th Fourier coefficient $a_{m, \mu}$ of $f$. These functionals generate $\left({\rm{Mod}}_{k,\Lambda}^\circ\right)^\vee$. The key theorem that allows us to study the effective cone is the following:

\begin{theorem}[\cites{Bor99, McG03, Bru02, Bru14, BLMM17, BZ24}]
\label{sec2:thm:Pic=Mod}
Let $\Lambda$ be an even lattice of signature $(2,n)$ with $n\geq3$ splitting off two copies of $U$. Then the map 
\begin{equation}
\label{sec2:eq:isoM=P}
\varphi\colon \left({\rm{Mod}}_{k,\Lambda}^\circ\right)^\vee\longrightarrow {\rm{Pic}}_{\mathbb{Q}}\left(\mathcal{D}\big/\widetilde{{\rm{O}}}^+(\Lambda)\right),\;\;\; c_{m,\mu}\mapsto H_{-m,\mu}
\end{equation}
is an isomorphism of $\mathbb{Q}$-vector spaces for $k=1+n/2$. 
\end{theorem}

\begin{remark}\label{rem: hodge class}
    Under the above isomorphism $\varphi$,  the Hodge class $\lambda$ is identified with the functional $-c_{0,0}$ sending $E_{k,\Lambda}$ to $-1$ and $S_{k,\Lambda}$ to $0$.
\end{remark}

The fact that $\varphi$ is a well-defined $\mathbb{Q}$-homomorphism follows from \cite{Bor99,McG03}, injectivity follows from \cite{Bru02}*{Theorem 0.4} and  \cite{Bru14}*{Theorem 1.2}, and surjectivity is Theorem \ref{thm:sec2:BLMM17}.

\subsection{Effective and NL cones}\label{section: NL cone prelim}

It was shown in \cite{BM19} that, on the left-hand side of \eqref{sec2:eq:isoM=P}, the functionals $c_{m,\mu}$ converge projectively to $-c_{0,0}$ as $m$ grows. This  implies that the cone spanned by all $H_{m,\mu}$ is polyhedral. Using the formula \eqref{sec2:eq:H-P}, Bruinier--M\"oller moreover show that the cone ${\rm{Eff}}^{NL}\left(\mathcal{D}\big/\widetilde{\rm{O}}^+\left(\Lambda\right)\right)$ generated by primitive Heegner divisors $P_{\Delta,\delta}$ is polyhedral, answering  \cite{Pet15}*{Question 4.5.2}. More precisely, \cite{BM19} shows that there is a neighborhood $\mathcal{U}$  of $\mathbb{Q}_{\geq0}\lambda$ strictly contained in ${\rm{Eff}}^{NL}\left(\mathcal{D}\big/\widetilde{\rm{O}}^+\left(\Lambda\right)\right)$ and a value $\Delta_0$ such that  for all $\Delta\geq \Delta_0$, we have $P_{\Delta,\delta}\in \mathcal{U}$.
The NL cone ${\rm{Eff}}^{NL}\left(\mathcal{D}\big/\widetilde{\rm{O}}^+\left(\Lambda\right)\right)$ is then the convex hull of the divisors $P_{\Delta,\delta}$ for $\Delta\le \Delta_0$.

Formulas for the NL cones ${\rm{Eff}}^{NL}\left(\mathcal{F}_{2d}\right)$ for low values of $d$ were conjectured in \cite{Pet15} by looking at truncated Fourier coefficients of the modular forms generating ${\rm{Mod}}_{\frac{21}{2},\Lambda_{2d}}^\circ$, see \cite{Pet15}*{Remark 4.7.1}. More precisely, for $d\leq 18$ Peterson used \eqref{sec2:eq:isoM=P} to compute the cone generated by the
$8d$ generators $P_{Q(\delta)+j,\delta}$ for $\delta\in D(\Lambda_{2d})$ and $j\in \{0,1,2,3\}$. He then conjectured that this cone coincides with ${\rm{Eff}}^{NL}\left(\mathcal{F}_{2d}\right)$ for these values of $d$. 

Confirming these formulas for a given $d$ requires explicitly computing the $\mathcal{U}$ and $\Delta_0$ described above.  This has to do with finding concrete bounds analogous to {\textit{Deligne's bound}} for scalar-valued Hecke eigenforms of integral weight. Once these $\mathcal{U}$ and $m_0$ are computed, calculating  ${\rm{Eff}}^{NL}\left(\mathcal{D}\big/\widetilde{\rm{O}}^+\left(\Lambda\right)\right)$ can be accomplished by computer. See Section \ref{sec: NL cone comp} for more details. 

Let $X$ be a normal $\mathbb{Q}$-factorial quasi-projective variety with ${\rm{Pic}}_{\mathbb{Q}}\left(X\right)$ a finite dimensional $\mathbb{Q}$-vector space. The {\textit{effective cone}} ${\rm{Eff}}\left(X\right)$ is the cone in ${\rm{Pic}}_{\mathbb{Q}}\left(X\right)$ generated by all effective $\mathbb{Q}$-divisors up to linear equivalence: 
\[
{\rm{Eff}}\left(X\right)=\left\langle E\in {\rm{Pic}}_{\mathbb{Q}}\left(X\right)\left|E \hbox{ is effective}\right.\right\rangle_{\mathbb{Q}_{\geq0}}.
\]
When $X$ is projective and $h^1(X,\mathcal{O}_X)=0$, then ${\rm{Pic}}_{\mathbb{Q}}\left(X\right)$ coincides with the Neron-Severi group ${\rm{NS}}(X)_{\mathbb{Q}}$ and one recovers the standard definition. The definition for $\mathbb{R}$-divisors is the same. Further, the cone is often not closed and the closure is called the {\textit{pseudo-effective cone}}, denoted $\overline{{\rm{Eff}}}\left(X\right)$.

\section{Generators of the NL cone}\label{sec: NL cone comp}

Throughout this section, we assume that $\Lambda$ is a lattice of signature $(2,n)$ with $n\ge 3$ splitting off one copy of the hyperbolic plane. We moreover consider the half-integer $k=1+n/2$.

As described in Section \ref{section: NL cone prelim}, in order to describe the NL cone ${\rm{Eff}}^{NL}\left(\mathcal{D}\big/\widetilde{\rm{O}}^+\left(\Lambda\right)\right)$ for a given lattice $\Lambda$, one needs to calculate a neighborhood $\mathcal{U}$ of $\mathbb{Q}_{\geq0}\lambda$ strictly contained in ${\rm{Eff}}^{NL}\left(\mathcal{D}\big/\widetilde{\rm{O}}^+\left(\Lambda\right)\right)$ and an explicit value $\Omega$ such that $P_{\Delta,\delta}\in \mathcal{U}$ for all $\Delta > \Omega$.  Further, as in Subsection \ref{ss:Pic}, we view ${\rm{Mod}}_{k,\Lambda}$ and $S_{k,\Lambda}$ as $\mathbb{Q}$-vector spaces.

In order to find such an explicit $\Omega$, we fix a rational basis $\{f_1,\ldots f_M\}$ for $S_{k, \Lambda}$. Let $e = E_{k, \Lambda}$ be the Eisenstein series defined in \eqref{sec2:eq:E}. We use the isomorphism \eqref{sec2:eq:isoM=P} to identify each $H_{-m, \mu}$ with the coefficient functional $c_{m, \mu}$ and hence a tuple \[c_{m,\mu}(e,f_1, \ldots, f_M)=(c_{m,\mu}(e), c_{m,\mu}(f_1), \ldots, c_{m,\mu}(f_M))\in \mathbb{Q}^{M+1}.\] 
Intuitively, as $m$ increases, the coefficients $c_{m, \mu}(e)$ of $E_{k, \Lambda}$ grow more rapidly than those of any cusp form, and therefore $c_{m, \mu}(e, f_1, \ldots, f_M)$ converges projectively to $(-1,0,\ldots, 0)$, which corresponds to the Hodge class $\lambda$ (see Remark \ref{rem: hodge class}). This convergence is proved in \cite{BM19}*{Proposition 4.5}. However to produce the required neighborhood $\mathcal{U}$ and bound $\Omega$, we need to make this convergence quantitative: we need explicit upper bounds for the Fourier coefficients of vector-valued cusp forms and an explicit lower bound for the coefficients of the Eisenstein series.

The coefficients of $e$ can be expressed in closed form \cite{BK01} and a lower bound of the form $c_{m, \mu}(e) \ge C_{k, \Lambda} \cdot m^{k-1}$, where $C_{k,\Lambda}$ is an explicit positive constant depending only on the lattice $\Lambda$ and weight $k$, easily follows, see \cite{BM19}*{Propositions 3.2 and 4.5}. As for cusp forms, despite the considerable literature on bounds for the growth of Fourier coefficients, we are unaware of a general bound (with explicit constants) that applies to our situation so we derive one below. The bound we derive is only the trivial bound $\mathrm{O}(m^{k/2})$, but this is sufficient to distinguish it from the growth of the lower bound for $c_{m,\mu}(e)$.

We will use the fact that the space of cusp forms $S_{k, \Lambda}$ is spanned by Poincar\'e series
\[P_{k,(m, \mu)}(\tau)=\frac{1}{2}\sum_{\substack{c,d\in \mathbb{Z}, \\{\rm{gcd}}(c,d)=1}} (c\tau+d)^{-k}e^{2\pi i m \frac{a\tau+b}{c\tau+d}}\rho_\Lambda\left(\left(\begin{array}{cc}a&b\\ c&d\end{array}\right)\right)\mathfrak{e}_\mu.\]
These are characterized through the Petersson inner product 
\[
\langle f, g \rangle := \int_{\mathrm{SL}_2(\mathbb{Z}) \backslash \mathbb{H}} \sum_{\mu \in D(\Lambda)} f_{\mu}(\tau) \overline{g_{\mu}(\tau)} y^k \, \frac{\mathrm{d}x \, \mathrm{d}y}{y^2}, \quad f, g \in S_{k, \Lambda}
\]
by the fact that they represent (up to a constant factor) the coefficient extraction functionals: an arbitrary cusp form 
\begin{equation}
\label{sec3:eq:cusp}
f(\tau)=\sum_{\mu\in D(\Lambda)}\sum_{m\in \frac{1}{N}\mathbb{Z}_{>0}}a_{m,\mu}q^m\mathfrak{e}_\mu
\end{equation}
has Fourier coefficients $a_{m, \mu}$ which can be written
\begin{equation}
\label{sec3:eq:Pprod}
a_{m, \mu}=\frac{(4\pi m)^{k-1}}{\Gamma(k-1)}\left\langle f, P_{k, (m, \mu)}\right\rangle.
\end{equation}

This implies that to bound the coefficients of arbitrary cusp forms, it is sufficient to bound the growth of the ``diagonal'' coefficients of Poincar\'e series. More precisely:

\begin{lemma}\label{lem: cusp form bound} Suppose the coefficients of $$P_{k, (m, \mu)}(\tau) = \sum_{\beta \in D(\Lambda)} \sum_{n \in Q(\beta)+\mathbb{Z}} c_{m, \gamma}(n, \beta) q^n \mathfrak{e}_{\beta}$$ satisfy a bound of the form $$|c_{m, \mu}(m, \mu)| \le C \cdot m^A$$
for some positive constants $A$ and $C$. Then the coefficients of every cusp form \eqref{sec3:eq:cusp} satisfy the bound 
\[
\left|a_{m,\mu}\right| \le \tilde C \cdot m^{A/2 + (k-1)/2} \cdot \|f\|
\]
with constant 
\[
\tilde C := \frac{(4\pi)^{(k-1)/2}}{\sqrt{\Gamma(k-1)}} \cdot \sqrt{C}.
\]
\end{lemma}
\begin{proof} From \eqref{sec3:eq:Pprod} it follows that the Petersson norm of $P_{k, (m, \mu)}$ is 
\begin{align*} \|P_{k, (m, \mu)}\| &= \sqrt{ \langle P_{k, (m, \mu)}, P_{k, (m, \mu)} \rangle} = \frac{\sqrt{\Gamma(k-1)}}{(4\pi m)^{(k-1)/2}} \cdot |c_{m, \mu}(m, \mu)|^{1/2} \\ &\le \sqrt{C \cdot \Gamma(k-1)}{(4\pi)^{k/2-1/2}} \cdot m^{A/2 + (1-k)/2}. \end{align*} The Cauchy--Schwarz inequality then yields 
\begin{align*} |a_{m,\mu}| &= \frac{(4\pi m)^{k-1}}{\Gamma(k-1)} |\langle f, P_{k, (m, \mu)} \rangle| \\ & \le \frac{(4\pi m)^{k-1}}{\Gamma(k-1)} \cdot \|f\| \cdot \|P_{k, (m, \mu)}\| \\ &\le \frac{(4 \pi)^{(k-1)/2} \sqrt{C}}{\sqrt{\Gamma(k-1)}} m^{A/2+(k-1)/2} \cdot \|f\|. \qedhere \end{align*}
\end{proof}

The following lemma gives an explicit bound of the form required in Lemma 3.1.
\begin{lemma} For any half-integer $k \ge 5/2$, the diagonal coefficients $c_{m, \mu}(m, \mu)$ of $P_{k, (m, \mu)}$ satisfy $$|c_{m, \mu}(m, \mu)| \le C' \cdot m + 1$$ with constant \[
C' = C'(k) = \frac{(2\pi)^k}{\Gamma(k) \cdot (k-2)} + 2.125.
\]
\end{lemma}

For $m \ge 1$, we therefore have $|c_{m,\mu}(m,\mu)| \le Cm$ with $C = C'+1$. Lemma 3.1 then yields $$|a_{m,\mu}| \le \tilde C m^{k/2} \cdot \|f\|.$$ 
\begin{proof}

From \cite{Bru02}*{Chapter 1.2}, the Fourier coefficients of \[
P_{k, (m, \mu)} = \sum_{\beta \in D(\Lambda)} \sum_{n\in  Q(\beta)+\mathbb{Z}} c_{m, \mu}(n, \beta) q^n \mathfrak{e}_{\beta}
\]
are given by the formula 
\begin{align*}
c_{m, \mu}(n, \beta) &= 2\pi \Big( \frac{m}{n} \Big)^{(1 - k)/2} \sum_{c=1}^{\infty} \frac{1}{c} J_{k-1}(4\pi \sqrt{mn} / c) \cdot \mathrm{Re} \Big[ e^{-\pi i k} K_c(\mu, m, \beta, n) \Big] \\ &+ \frac{1}{2}\delta_{m=n,\mu=\beta} + \frac{1}{2}\delta_{m=n,\mu=-\beta},
\end{align*}
where $\delta$ is the Kronecker delta, $K_c$ is the generalized Kloosterman sum $$K_c(\mu, m, \beta, n) = \sum_{d \in (\mathbb{Z}/c\mathbb{Z})^{\times}} e^{2\pi i (ma + nd) / c} \langle \rho(M)^{-1} \mathfrak{e}_{\mu}, \mathfrak{e}_{\beta} \rangle,$$ and $J$ is the usual Bessel function. For our application, the trivial bound $|K_c(\mu, m, \beta, n)| \le c$ will be enough.

The Bessel function satisfies the bounds $$|J_{k-1}(x)| \le \frac{M}{x^{1/3}}, \quad \text{where} \; M \approx 0.78574687$$ (see \cite{Landau2000}) and $$|J_{k-1}(x)| \le \frac{x^{k-1}}{2^{k-1} \Gamma(k)}$$ (see \cite{NIST}, 10.14.4). For small values of $c$ (say $c \le n$), we use the first bound: \begin{align*} \Big| \sum_{c = 1}^n \frac{1}{c} J_{k-1}(4\pi \sqrt{mn}/c) \cdot \mathrm{Re} \Big[ e^{-\pi i k} K_c(\mu, m, \beta, n) \Big] \Big| &\le (4\pi \sqrt{mn})^{-1/3} M \cdot \sum_{c=1}^n c^{1/3} \\ &\le (4\pi)^{-1/3} m^{-1/6} M \cdot n^{7/6}. \end{align*}
We use the second bound for $c > n$: \begin{align*} \Big| \sum_{c > n} \frac{1}{c} J_{k-1}(4\pi \sqrt{mn}/c) \cdot \mathrm{Re} \Big[ e^{-\pi i k} K_c(\mu, m, \beta, n) \Big] \Big| &\le \frac{(2\pi)^{k-1}(mn)^{(k-1)/2}}{\Gamma(k)} \sum_{c > n} \frac{1}{c^{k-1}} \\ &\le \frac{(2\pi)^{k-1} m^{(k-1)/2} n^{(3-k)/2}}{\Gamma(k)(k-2)}, \end{align*}
where in the last step, we used $\sum_{c > n} c^{1-k} < \int_n^{\infty} \frac{\mathrm{d}t}{t^{k-1}} = \frac{n^{2-k}}{k-2}$. Altogether, we have \begin{align*} 
\left|c_{m, \mu}(n, \beta)\right| &\le 2\pi \Big( \frac{m}{n} \Big)^{(1-k)/2} \cdot (4\pi)^{-1/3} m^{-1/6} M \cdot n^{7/6} + \frac{(2\pi)^k}{\Gamma(k)(k-2)} n + 1.\end{align*}
For the diagonal coefficient $(m,\mu)=(n,\beta)$, we obtain 
\[
\left|c_{m, \mu}(m, \mu)\right| \le 2^{1/3} \pi^{2/3}M \cdot m + \frac{(2\pi)^k}{\Gamma(k)(k-2)} \cdot m + 1.
\]
The claim follows because $2^{1/3} \pi^{2/3} M < 2.125$.
\end{proof}

We now describe how to use the  bounds of Lemma \ref{lem: cusp form bound} to make the argument of \cite{BM19} explicit, thereby proving Theorem \ref{int:thm:1}. 

We will first describe how to compute the cone of Heegner divisors ${\rm Eff}^H\left(\mathcal{D}\big/\widetilde{\rm{O}}^+(\Lambda)\right)$. Let $\left({\rm{Mod}}_{k,\Lambda}^\circ\right)^{\vee}$ be the space of linear functionals on ${\rm{Mod}}_{k,\Lambda}^\circ$ and  consider the cone $\mathcal{C}$ generated by the coefficient extraction functionals 
 \[
 c_{m, \mu} : {\rm{Mod}}_{k,\Lambda}^\circ \longrightarrow \mathbb{Q}, \quad \sum a_{m, \mu} q^m \mathfrak{e}_{\mu} \mapsto a_{m, \mu}.
 \]
Write $c_{m, \mu} = \gamma_{m, \mu} e + s_{m, \mu}$, where $e$ is the functional 
\[
e(E_{k, \Lambda}) = -1, \quad \; e \Big|_{S_{k, L}} = 0,
\]
and $s_{m, \mu}(E_{k, \Lambda}) = 0.$ In particular, 
\[
E_{k, \Lambda}(\tau) = \mathfrak{e}_0 - \sum_{m, \mu} \gamma_{m, \mu} q^m \mathfrak{e}_{\mu}.
\]
  
We need to find an open neighborhood of $e$ contained in the cone $\mathcal{C}$. As in \cite{BM19}, there is a finite set of indices $(m_i, \mu_i)$, $1 \le i \le N$ and positive rationals $\lambda_i$ such that $c_{m_i, \mu_i}$ spans $S_{k, L}^{\vee}$ and 
\begin{equation}
\label{sec3:eq:lambda_i}
\sum_{i=1}^N \lambda_i c_{m_i, \mu_i} = e.
\end{equation}

Following \cite{BM19}*{Proposition 3.3}, the $\lambda_i$ can be constructed as follows. For $b$ sufficiently large (explicit) positive integer let $f$ be the weakly holomorphic modular form 
\[
f(\tau)=\Delta(\tau)^{-b}\cdot E_{(2-k) + 12b, \Lambda(-1)}(\tau),\;\;\hbox{where}\;\;
\Delta(\tau)=\eta(\tau)^{24}=q\cdot\prod_{n\geq1}(1-q^n)^{24}
\]
is the scalar-valued discriminant modular form. Recall that 
\[
\Delta(\tau)^{-b}=q^{-b}\cdot\left(\prod_{n\geq1}\frac{1}{1-q^n}\right)^{24b}=q^{-b}\cdot\left(\sum_{n=0}^{\infty}p(n)q^n\right)^{24b},
\]
where $p(n)$ is the number of partitions of $n$. In particular the coefficient of $q^{m}$ in the expansion of $\Delta^{-b}$ is zero for $m<-b$ and the Fourier coefficients of the product $f(\tau)$ can be computed explicitly. We write 
\[
f(\tau) = \sum_{\mu \in D(\Lambda(-1))} \sum_{m \in \mathbb{Q}} \alpha_{m, \mu} q^m \mathfrak{e}_{\mu}.
\]
As a consequence of the residue theorem one has that for any cusp form $f
\in S_{k,\Lambda}$, 
\[
\sum_{\substack{(m,\mu)\\-b\leq m<0}}
\alpha_{m,\mu}c_{-m,\mu}(f)=0
\] 
and we simply have to choose $b\geq\lceil k/12\rceil$ large enough such that the above functionals $c_{-m,\mu}$ span $S_{k,\Lambda}^{\vee}$. Then taking such a collection as a generating set and $\lambda_i = \frac{\alpha_{-m_i, \mu_i}}{\alpha_{0, 0}}$ with $m_i > 0$ one can ensure \eqref{sec3:eq:lambda_i} holds. This is the only input needed to produce a bound for a generating set of both the Heegner and the NL cones.

\begin{example}As an example, we take the lattice $\Lambda=\Lambda_4$ corresponding to the moduli of degree four K3 surfaces. Then $S_{\frac{21}{2},\Lambda_4}$ is two-dimensional generated by 
\begin{align*}
f_1=&\left(-128q - 57344q^2+\ldots\right)\mathfrak{e}_0+\left(q^{1/8} - 7q^{9/8}+\ldots\right)\mathfrak{e}_{\ell_*}\\
&+\left(4864q^{3/2} + 368640q^{5/2}+\ldots\right)\mathfrak{e}_{2\ell_*}
+\left(q^{1/8} - 7q^{9/8}+\ldots\right)\mathfrak{e}_{3\ell_*},\\
f_2=&\left(-14q - 568q^2+\ldots\right)\mathfrak{e}_0+\left(32q^{9/8} + 544q^{17/8}+\ldots\right)\mathfrak{e}_{\ell_*}\\
&+\left(q^{1/2} - 188q^{3/2}+\ldots\right)\mathfrak{e}_{2\ell_*}
+\left(32q^{9/8} + 544q^{17/8}+\ldots\right)\mathfrak{e}_{3\ell_*}.
\end{align*}
Here the dots mean higher-order terms. Since $\lceil k/12 \rceil = 1$, we take $b=1$. Then
\begin{align*}
\Delta^{-1}(\tau)&=q^{-1}\left(1+1q+2q^2+3q^3+5q^4+7q^5+\ldots\right)^{24}\\
&=q^{-1}+24+324q+ 3128q^2+\ldots
\end{align*}
and one obtains
\[
\Delta^{-1}\cdot E_{\frac{7}{2},\Lambda(-1)}=q^{-1}\mathfrak{e}_0+64q^{-1/8}\mathfrak{e}_{\ell_*}+14q^{-1/2}\mathfrak{e}_{2\ell_*}+64q^{-1/8}\mathfrak{e}_{3\ell_*}+108\mathfrak{e}_0+\sum_{\substack{(m,\mu)\\m>0}}\alpha_{m,\mu}q^{m}\mathfrak{e}_{\mu}.
\]
Recall that $c_{m,\mu}=c_{m,-\mu}$. One easily checks that the set of all $c_{m,\mu}$ with $0<m\leq 1$, in this case $\{c_{m_i,\mu_i}\}_{i=1}^4$ with indices 
\[
(m_1,\mu_1)=(1,0), (m_2,\mu_2)=(1/8,\ell_*), (m_3,\mu_3)=(1/2,2\ell_*), \hbox{and} (m_4,\mu_4)=(1/8,3\ell_*)
\]
generates $S_{k,\Lambda_4}^\vee$. Then with $\lambda_1=\frac{1}{108}, \lambda_2=\lambda_4=\frac{64}{108}$, and $\lambda_3=\frac{14}{108}$, Equation \eqref{sec3:eq:lambda_i} holds. 
\end{example}

We will need to rewrite these results for the Petersson norm in terms of the $\ell^2$-norm on $\mathbb{Q}^M$. Recall that we identify each functional $s_{m, \mu}$ with the tuple $$(s_{m, \mu}(f_1),..., s_{m,\mu}(f_M)) \in \mathbb{Q}^M$$ where $f_1,...,f_M$ is a rational basis of $S_{k, \Lambda}$. 

Define an inner product on ${\rm{Mod}}_{k,\Lambda}^\circ$ as follows: if $f\in S_{k,\Lambda}$ then $\|f\|$ is the usual Petersson norm and we declare the Eisenstein series $E_{k, L}$ to have norm one and be orthogonal to $S_{k,\Lambda}$.

To pass from $\|f\|$ to the $\ell^2$-norm $\|f\|_{\ell^2}$, we need a rational basis whose Petersson norms can be estimated explicitly. One such basis was described in \cite{Wil18}: 
\begin{equation}
\label{sec3:eq:sp_basis}
f_{m, \mu} := \sum_{\lambda = 1}^{\infty} P_{k, (\lambda^2 m, \lambda \mu)}.
\end{equation}
These are convenient because their Petersson norm is easy to bound using Lemma 3.1. Indeed, writing $f_{m, \mu} = \sum c(n, \gamma) q^n \mathfrak{e}_{\gamma}$, for any $m \ge 1$ one has \begin{align*} \left\|f_{m, \mu}\right\|^2 &\le \sum_{\lambda = 1}^{\infty} \frac{\Gamma(k-1)}{(4 \pi \lambda^2 m)^{k-1}} |c(\lambda^2 m, \lambda \mu)| \\ &\le \frac{\tilde C \cdot \Gamma(k-1)\cdot \|f_{m, \mu}\|}{(4\pi)^{k-1}} \sum_{\lambda = 1}^{\infty} \frac{(\lambda^2 m)^{k/2}}{(\lambda^2 m)^{k-1}} \\ &= \frac{\tilde C \cdot \Gamma(k-1) \cdot \zeta(k-2)}{(4\pi)^{k-1}} \cdot \|f_{m, \mu}\| \cdot m^{1-k/2}. \end{align*} 
Therefore, 
\[
\left\|f_{m, \mu}\right\| \le \frac{\tilde C \cdot \Gamma(k-1) \cdot \zeta(k-2)}{(4\pi)^{k-1}} \cdot m^{1-k/2}.
\]
So with respect to this basis, the Petersson norm and the $\ell^2$-norm on $\mathbb{Q}^M$ of $s_{m, \mu}$ are related by 
\begin{align}
\label{sec3:eq:P-ell2}
\begin{split}
\|s_{m, \mu}\| &= \sup_{f \ne 0} \frac{|s_{m, \mu}(f)|}{\|f\|} \\ &\ge \frac{1}{\max_i \|f_i\|} \sqrt{\frac{1}{M}\sum_{i=1}^M |s_{m, \mu}(f_i)|^2} \\ &\ge \frac{(4\pi)^{k-1} \cdot \max_i m_i^{k/2 - 1}}{\tilde C \cdot \Gamma(k-1) \zeta(k-2) \sqrt{M}} \cdot \|s_{m, \mu}\|_{\ell^2}. \end{split}\end{align}

Now we can bound the number of generators of the cone $\mathcal{C}$.

\begin{theorem}
\label{thm: nonprimitive cone generators} 
Assume $k>3$. Then for any choice of $\lambda_i$ and $m_i$ as above (see Equation \eqref{sec3:eq:lambda_i}), the cone $\mathcal{C}$ generated by all coefficient functionals is already generated by $c_{m, \mu}$ with 
\[
m \le \Big( \frac{R \cdot C_{k, \Lambda}}{B} \Big)^{2 / (2-k)},
\]
where $C_{k, \Lambda}$ is any constant such that the Fourier coefficients $e(m, \mu)$ of $E_{k, \Lambda}$ are bounded from below by \[
\left|e_{m, \mu}\right| \ge C_{k, \Lambda} \cdot m^{k-1},
\]
where $R > 0$ is such that the convex hull $\mathcal{C}_S$ of $\frac{s_{m_i, \mu_i}}{\gamma_{m_i, \mu_i}}$ contains the ball of radius $R$ with respect to the $\ell^2$-norm, and where 
\begin{equation}
\label{sec3:eq:B}
B := \frac{(\tilde C)^2 \Gamma(k-1) \zeta(k-2) \sqrt{M}}{(4\pi)^{k-1} \cdot \max_i m_i^{k/2 - 1}}
\end{equation}
where $\tilde C$ is the constant from Lemma 3.1.
\end{theorem}

\begin{remark}
Note that $\mathcal{C}_S$ contains an open neighborhood of $0$ by \cite{BM19}. To compute a concrete radius $R$, we write $\mathcal{C}_S\subset \mathbb{Q}^M$ as an intersection of finitely many half-planes, say $\{x: \, \langle v, x \rangle \le a\}$, and take $R$ to be the minimum of $|a| / \|v\|_{\ell^2}$, where the latter is the standard $\ell^2$-norm on $\mathbb{Q}^M$.  When the discriminant of $\Lambda$ is $D$, it can be derived from \cite{BM19} that the constant $C_{k,\Lambda}$ can be chosen to be 
\[
C_{k, \Lambda} = \frac{16}{5} \Big( \frac{\pi}{2} \Big)^k \cdot \frac{\sqrt{D}}{\zeta(k - 1/2) \Gamma(k)} \prod_{\substack{\text{primes} \\ p | D}} \frac{1 - 1/p}{1 - 1 / p^{2k-1}}.
\]
As an example, for the lattices $\Lambda = \Lambda_d$ and $k = 21/2$, this bound is approximately $$C_{k, \Lambda} \approx 0.0002286 \cdot \sqrt{d} \prod_{\substack{p | d \\ p \, \text{odd}}} \frac{1 - 1/p}{1 - 1/p^{20}}.$$
\end{remark}
\begin{proof}[Proof of Theorem \ref{thm: nonprimitive cone generators}]
The coefficient functional $s_{m, \mu}$ is bounded in operator norm by $$\|s_{m, \mu}\| \le \tilde C \cdot m^{k/2}$$ by Lemma 3.1, and therefore in $\ell^2$-norm by $$\|s_{m, \mu}\|_{\ell^2} \le B \cdot m^{k/2}$$ with the constant $B$ given by \eqref{sec3:eq:B}. Recall that $\|e\| = 1.$ Since $C_{k,L}$ is such that $$\gamma_{m, \mu} \ge C_{k, \Lambda} \cdot m^{k-1},$$ we have \[\Big\| \frac{c_{m, \mu}}{\gamma_{m, \mu}} - e \Big\|_{\ell^2} \le \frac{B}{C_{k, \Lambda}} \cdot m^{1 - k/2}. \]
Therefore, if $\frac{B}{C_{k, \Lambda}} m^{1-k/2} < R$ then $c_{m, \mu}$ belongs to the interior of $\mathcal{C}$. 

\end{proof}

Note that when $k=3$, the constant $B$ in \eqref{sec3:eq:B} diverges, making the bound of Theorem \ref{thm: nonprimitive cone generators} equal to infinity. To obtain a meaningful bound in this case we have to refine our methods.

\begin{theorem}
For integral $k>2$ (in particular $k=3$), the same as in Theorem \ref{thm: nonprimitive cone generators} holds with 
\begin{equation}
\label{sec3:them:Bk=3}
B:= \frac{\tilde C \cdot N^{(k-1)/2} C_{k,N^2} \Gamma(k-1) \zeta(k-1)^3 \sqrt{M}}{(4\pi)^{k-1} \zeta(2k-2) \cdot \max_i \sigma_0(Nm) m_i^{(k-1)/2}},
\end{equation}
where $N$ is the level of $\Lambda$, i.e., the smallest positive integer such that $N\cdot q_\Lambda=N\cdot\frac{\langle\cdot,\cdot\rangle}{2}$ is integral on $\Lambda^{\vee}$ (equivalently trivial on $D(\Lambda)$), 
\[
C_{k,N^2}=2N\cdot \sqrt{\pi}\cdot e^{2\pi}\prod_{p\mid N}\frac{(1+1/p)^3}{\sqrt{1-1/p^4}}\cdot\sqrt{\dim S_k\left(\Gamma_1(N^2)\right)},
\]
and $S_k\left(\Gamma_1(N^2)\right)$ is the standard space of scalar-valued cusp forms of weight $k$ with trivial character for $\Gamma_1(N^2)\subset {\rm{SL}}_2(\mathbb{Z})$.
\end{theorem}

\begin{proof}
From \cite{S-PY18}*{Theorem 12} one has that any scalar-valued cusp form 
\[
f(\tau) = \sum_{n=1}^{\infty} a(n) q^n \in S_k(\Gamma_0(N), \chi)
\]
satisfies the coefficient bound 
\[
|a(n)| \le C_{k,\chi} \cdot\|f\| \cdot \sigma_0(n) n^{(k-1)/2}
\]
with the constant 
\[
C_{k,\chi} = 2 \sqrt{N\pi} e^{2\pi} \prod_{p | N} \frac{(1 + 1/p)^3}{\sqrt{1 - 1/p^4}} \cdot \sqrt{\mathrm{dim}\, S_k(\Gamma_0(N), \chi)}.
\]
A slight variation of their proof shows that when $f(\tau) \in S_k(\Gamma_1(N))$, one has
\[
|a(n)| \le C_{k,N} \cdot \|f\| \cdot \sigma_0(n) n^{(k-1)/2}
\]
with 
\[
C_{k,N} = 2 \sqrt{N\pi} e^{2\pi} \prod_{p | N} \frac{(1 + 1/p)^3}{\sqrt{1 - 1/p^4}} \cdot \sqrt{\mathrm{dim}\, S_k(\Gamma_1(N))}.
\]

If $f = (f_{\mu})_{\mu \in D(\Lambda)} \in S_{k,\Lambda}$ is a vector-valued cusp form attached to $\Lambda$ of even rank (and therefore integer weight) and level $N$, then $f_{\mu}(N\tau)$ belongs to $S_k(\Gamma_1(N^2))$ for every $\mu \in D(\Lambda)$. Hence the coefficients $a_{m,\mu}$ of $f$ satisfy 
\[
|a_{m,\mu}| \le N^{(k-1)/2}\cdot C_{k, N^2}\cdot \|f\| \sigma_0(Nm) \cdot m^{(k-1)/2}.
\]
Using this bound instead of Lemma \ref{lem: cusp form bound} for the series $f_{m, \mu}$, we obtain 
\begin{align*} 
\|f_{m,\mu}\| &\le \|f_{m,\mu}\|^{-1} \sum_{\lambda=1}^{\infty} \frac{\Gamma(k-1)}{(4\pi \lambda^2 m)^{k-1}} |c(\lambda^2 m, \lambda \mu)| \\ 
&\le N^{(k-1)/2} \cdot \frac{C_{k, N^2}\cdot \Gamma(k-1)}{(4\pi)^{k-1}} \sum_{\lambda = 1}^{\infty} \frac{\sigma_0(\lambda^2 Nm) (\lambda^2 m)^{(k-1)/2}}{(\lambda^2 m)^{k-1}} \\ 
&\le \frac{N^{(k-1)/2} \cdot C_{k, N^2}\cdot \Gamma(k-1) }{(4\pi)^{k-1}} \sigma_0(Nm) m^{(1-k)/2}\sum_{\lambda = 1}^{\infty} \frac{\sigma_0(\lambda^2)}{\lambda^{k-1}} \\ 
&= \frac{N^{(k-1)/2} \cdot C_{k, N^2} \cdot \Gamma(k-1) \zeta(k-1)^3}{(4\pi)^{k-1} \zeta(2k-2)} \cdot \sigma_0(Nm) m^{(1-k)/2}.
\end{align*}

Here we use the submultiplicativity $\sigma_0(mn) \le \sigma_0(m)\sigma_0(n)$ and the elementary Dirichlet series identity 
\[
\sum_{n=1}^{\infty} 
\frac{\sigma_0(n^2)}{n^s} = \frac{\zeta(s)^3}{\zeta(2s)}.
\]
Finally, the Petersson norm and the $\ell^2$-norm on $\mathbb{Q}^M$ also satisfy the inequality 
\[
\|s_{m,\mu}\| \ge \frac{(4\pi)^{k-1} \zeta(2k-2) \cdot \max_i \sigma_0(Nm) m^{(k-1)/2}}{N^{(k-1)/2} C_{k,N^2}\Gamma(k-1) \zeta(k-1)^3 \sqrt{M}} \cdot \|s_{m,\mu}\|_{\ell^2}.
\] 
Now the proof is the same as for Theorem \ref{thm: nonprimitive cone generators} taking $B$ as in \eqref{sec3:them:Bk=3}.
\end{proof}

We now impose the added assumption that $\Lambda$ splits off an additional copy of the hyperbolic plane. Then  the functionals $c_{m, \mu}$ correspond to the (non-primitive) Heegner divisors $H_{m, \mu}$ under the isomorphism \eqref{sec2:eq:isoM=P} of Theorem \ref{sec2:thm:Pic=Mod}. Hence in this case,  Theorem \ref{thm: nonprimitive cone generators} describes a generating set for the Heegner cone ${\rm Eff}^H\left(\mathcal{D}\big/\widetilde{\rm{O}}^+(\Lambda)\right)$. 

Continuing with the hypothesis that $\Lambda$ splits off not just one but two copies of the hyperbolic plane, we will now use the bounds of Theorem \ref{thm: nonprimitive cone generators} in order to compute the NL cone ${\rm Eff}^{NL}\left(\mathcal{D}\big/\widetilde{\rm{O}}^+(\Lambda)\right)$. 

To state the explicit bound $\Omega$ in the case of the $P_{\Delta, \delta}$ generating ${\rm Eff}^{NL}\left(\mathcal{D}\big/\widetilde{\rm{O}}^+(\Lambda)\right)$, define the functionals $$p_{\Delta, \delta} := \sum_{\substack{r \in \mathbb{Z}_{>0} \\ r^2 | \Delta}} \mu(r) \sum_{\substack{\sigma \in D(\Lambda) \\ r \sigma = \delta}} c_{\Delta / r^2, \sigma}.$$ Then, applying the isomorphism \eqref{sec2:eq:isoM=P}, one has $\varphi(p_{\Delta, \delta}) = P_{\Delta, \delta}$ is the corresponding primitive Heegner divisor by Equation \eqref{sec2:eq:H-P}. Let $\mathcal{P}$ be the cone generated by the $p_{\Delta, \delta}$. As in the case of the Heegner cone, using the isomorphism of Theorem \ref{sec2:thm:Pic=Mod}, a description of the generators of $\mathcal{P}$ gives a description of the generators of ${\rm Eff}^{NL}\left(\mathcal{D}\big/\widetilde{\rm{O}}^+(\Lambda)\right)$. 

\begin{theorem}\label{thm: NL cone generators} Let $B, C_{k, \Lambda}$ and $R$ be the constants of Theorem 3.4 and assume  $\Lambda$ has discriminant $D$ and splits off two copies of the hyperbolic plane. The cone $\mathcal{P}$ is already generated by $p_{\Delta, \delta}$ with $$\Delta \le \left( \frac{R \cdot C_{k, \Lambda} \cdot M}{B \cdot (1 + D\cdot (\zeta(k) - 1))^2} \right)^{2 / (2 - k)},$$ where $$M := 1 - \frac{1}{2} \Big( \prod_{p \, \text{prime}} \Big(1 + \frac{1}{p(p-1)} \Big) - \prod_{p \, \text{prime}} \Big(1 - \frac{1}{p(p-1)} \Big) \Big) > 0.215.$$
\end{theorem}
\begin{proof} By Lemma 3.1 and the triangle inequality, for any cusp form $f$, we have \begin{align*} |p_{\Delta, \delta}(f)| &\le \tilde C \cdot\|f\| \cdot \sum_{\substack{r \in \mathbb{Z}_{>0} \\ r^2 | \Delta}}  \sum_{\substack{\sigma \in D(\Lambda) \\ r \sigma = \delta}} \Big( \frac{\Delta}{r^2} \Big)^{k/2} \\ &\le \tilde C \cdot \|f\|\cdot \Delta^{k/2} \cdot\sum_{r=1}^{\infty} r^{-k} \cdot\left|\{\sigma \in D(\Lambda): \; r \sigma = 0\}\right| \\ &\le \tilde C\cdot \Delta^{k/2} \cdot \|f\| \cdot \Big(1 + D\cdot (\zeta(k) - 1) \Big).\end{align*}
On the other hand, if $E_{k,\Lambda}$ denotes the Eisenstein series then the proof of \cite{BM19}*{Proposition 4.5} shows that $$\left|p_{\Delta, \delta}(E_{k,\Lambda})\right| \ge \left|c_{\Delta, \delta}(E_{k,\Lambda})\right| \cdot M$$ with the constant $M$ defined above. 

So we can copy the proof of Theorem 3.4, with the upper and lower bounds for $c_{m, \mu}$ replaced by those for $p_{\Delta, \delta}$: we multiply $C_{k, \Lambda}$ by $M$ and $\tilde C$ (as part of the constant $B$) by $\zeta(k)\cdot D$.
\end{proof}

\begin{example}
    Continuing Example 3.3, the special basis \eqref{sec3:eq:sp_basis} for $S_{\frac{21}{2}, \Lambda_4}$ consists of the series \begin{align*} f_{1/8, \ell_*} &= \frac{7159053}{14318102} f_1 + \frac{7683852}{7159051} f_2 \\ f_{1/2, 2\ell_*} &= \frac{1}{7159051} f_1 + \frac{209563208}{221930581} f_2.\end{align*} With respect to this basis, the convex set $\mathcal{C}_S$ is the triangle with vertices $$\left(\frac{-7159053}{4}, -\frac{1}{2}\right), \left(-\frac{1}{2}, -\frac{3880799}{602547}\right), \left(\frac{2143005}{2873041}, \frac{122245370}{979706981}\right).$$ This triangle can be described by the inequalities $x \in \mathbb{R}^2$ with $$\left\langle (2, 602547), x \right\rangle \ge -3880800, \; \langle (38, -108856407), x \rangle \ge -13582800,$$ $$\langle (-3175198, 602547), x \rangle \ge -2293200,$$ so we obtain the radius $$R = \min \left( \frac{3880800}{\|(2, 602547)\|}, \frac{13582800}{\|(38, -108856407)\|}, \frac{2293200}{\|(-3175198, 602547)\|} \right) \approx 0.1248$$ for the largest circle centered at zero and contained in the given triangle.
\end{example}

We have implemented the Sage package \cite{heegner_cones}, which applies the method described above together with the bounds of Theorem \ref{thm: NL cone generators} in order to compute the NL cone ${\rm{Eff}}^{NL}\left(\mathcal{D}\big/\widetilde{\rm{O}}^+\left(\Lambda\right)\right)$.  

The bounds above are far from being sharp. For example, with $k = 21/2$ and $\Lambda = \Lambda_d$, $d \le 10$, the upper bound for $\Delta$ in Theorem 3.7 is given in the following table (rounded to three decimal places):

\begin{table}[hbt!]
\begin{tabular}{|l|l|l|l|l|l|}
\hline
      & $d=1$     & $d=2$     & $d=3$     & $d=4$     & $d=5$     \\ \hline
Bound & $133.378$ & $102.512$ & $111.564$ & $111.197$ & $120.525$ \\ \hline
      & $d=6$     & $d=7$     & $d=8$     & $d=9$     & $d=10$    \\ \hline
Bound & $131.690$ & $120.852$ & $125.560$ & $144.157$ & $141.558$ \\ \hline
\end{tabular}
\end{table}

On the other hand, in all cases we were able to compute, the cone of primitive Heegner divisors is already generated in discriminant $\Delta \le 2$. As a practical matter, we found it far more efficient to compute the cone generated by Heegner divisors with $\Delta \le 2$ and then check afterwards that it contains all $P_{\Delta, \delta}$ with $\Delta$ up to the above bound.

We now explicitly compute ${\rm{Eff}}^{NL}\left(\mathcal{D}\big/\widetilde{\rm{O}}^+\left(\Lambda\right)\right)$ in some key examples. 

 \subsection{Moduli of K3 surfaces}\label{sec: K3}

In this case  $\Lambda_{2d}=U^{\oplus 2}\oplus E_{8}(-1)^{\oplus 2}\oplus A_1(-d)$ and the quotient $\mathcal{F}_{2d}=\mathcal{D}/\widetilde{\rm{O}}^+\left(\Lambda_{2d}\right)$ is the moduli space of quasi-polarized K3 surfaces of degree $2d$. The method of Section \ref{sec: NL cone comp}
together with the bounds of Theorem \ref{thm: NL cone generators} (and their Sage implementation \cite{heegner_cones}) yield minimal generating rays of  ${\rm{Eff}}^{NL}\left(\mathcal{F}_{2d}\right)$ for low $d$. 
These calculations confirm the predictions of \cite{Pet15}*{Remark 4.7.1 and Table 4.5} (aside from one additional generator in the case $d=13$). We record these generators in Table \ref{sec6:table:NLK3} for $1\le d\le 20$.

\subsection{Hyperk\"{a}hler fourfolds of ${\rm{K3}}^{[2]}$-type}

Let $(X,L)$ be a primitively polarized hyperk\"{a}hler fourfold of ${\rm{K3}}^{[2]}$-type. The Beauville--Bogomolov--Fujiki lattice $(H^2(X,\mathbb{Z}),q_X)$ is isomorphic to \[\Lambda=U^{\oplus 3}\oplus E_8(-1)^{\oplus 2}\oplus A_1(-1).\]
The polarization $L$ comes with two invariants singling out a component of the moduli space. These are the Beauville--Bogomolov--Fujiki degree $2d$ and the divisibility $\gamma\in\{1,2\}$. Further, when $\gamma=2$, then $d=4t-1$ for some $t\geq1$. We will denote by $\mathcal{M}_{{\rm{K3}}^{[2]},2d}^\gamma$ the partial compactification of the corresponding moduli space given by the modular variety $\mathcal{D}\big/{\rm{Mon}}^2(\Lambda,h)$, where after choosing a marking, $\Lambda_h$ is the orthogonal complement of $h=c_1(L)$ in $\Lambda$, and ${\rm{Mon}}^2(\Lambda,h)=\widetilde{\rm{O}}^+\left(\Lambda_h\right)$, cf. \cite{Mar11}*{Lemma 9.2} and \cite{BBBF23}*{Proposition 3.7}.

We will denote by $\Lambda_d$ and $\Lambda_t$ the lattices $U^{\oplus 2}\oplus E_{8}(-1)^{\oplus 2}\oplus Q_d$ (resp. $Q_t$) where
\[
Q_d=\mathbb{Z}\ell+\mathbb{Z}\delta=\left(\begin{array}{cc}-2d&0\\0&-2\end{array}\right)\;\;\;\hbox{and}\;\;\;Q_t=\mathbb{Z} u+ \mathbb{Z}v=\left(\begin{array}{cc}-2t&1\\1&-2\end{array}\right).
\]

These correspond to the lattice $\Lambda_h$ for $(X,L)$ in $\mathcal{M}_{{\rm{K3}}^{[2]},2d}^\gamma$ when $\gamma=1$, respectively $\gamma=2$ with $d=4t-1$. When $\gamma=1$, the discriminant group is isomorphic to $\mathbb{Z}\big/2d\mathbb{Z}\times\mathbb{Z}\big/2\mathbb{Z}$, generated by $\ell_*$ and $\delta_*$. When $\gamma=2$, the discriminant group is isomorphic to $\mathbb{Z}/d\mathbb{Z}$ and is generated by $(2u+v)_*$.

The Sage program \cite{heegner_cones} following the procedure described above yields the minimal generators of the  NL cone ${\rm{Eff}}^{NL}\left(\mathcal{M}_{{\rm{K3}}^{[2]},2d}^\gamma\right)$  in the split case (the case $\gamma=1$) in Table \ref{sec6:table:NLK32_sp} for $d\leq 5$ and in the non-split case (the case $\gamma=2$) in Table \ref{sec6:table:NLK32_nsp} for $t\leq 5$ with $d=4t-1$.

\begin{remark}Double EPW sextics are particular ramified double covers of certain singular sextic hypersurfaces in $\mathbb{P}^5$. When smooth, they are hyperk\"ahler fourfolds of ${\rm{K3}}^{[2]}$-type with polarization of degree $2$ and divisibility $1$ and hence are elements of $\mathcal{M}_{{\rm{K3}}^{[2]},2}^1$. Interestingly, the generators $P_{-1,0}, P_{-\frac{1}{4},\ell_*}, P_{-\frac{1}{4},\delta_*}, P_{-\frac{1}{2},\ell_*+\delta_*}$ of the NL cone ${\rm{Eff}}^{NL}\left(\mathcal{M}_{{\rm{K3}}^{[2]},2}^1\right)$, together with the additional primitive Heegner divisor $P_{-\frac{5}{4},\delta_*}$, are precisely the divisorial part of the complement of the image under the period map of the subset of $\mathcal{M}_{{\rm{K3}}^{[2]},2}^1$ of double EPW sextics. See \cites{OG15, OG16}, \cite{DM19}*{Example 6.3}.    
\end{remark}

\subsection{Cubic Fourfolds}\label{sec: special cubic fourfolds} If $Y\subset \mathbb{P}^5$ is a smooth cubic fourfold, then its primitive cohomology $H^4(Y,\mathbb{Z})^\circ$ together with its intersection form is isomorphic to $\Lambda(-1)=U^{\oplus 2}\oplus E_{8}^{\oplus 2}\oplus A_2$, which, up to sign, is the same as above with $t=1$. Thus the computation of the NL cone of the partial compactification $\mathcal{D}/\widetilde{\rm{O}}^+\left(\Lambda\right)$ (see \cite{Voi86}) of the moduli space of smooth cubic fourfolds has been carried out in Table \ref{sec6:table:NLK32_nsp} (with $d=3$ and $\gamma=2$).

\section{Uniruledness}\label{sec: uniruled}
For the orthogonal modular variety $X=\mathcal{D}\big/\widetilde{{\rm{O}}}^+(\Lambda)$ consider the degree map \[{\rm{Pic}}_{\mathbb{Q}}\left(X\right)\longrightarrow \mathbb{Q}\]  given by taking the degree of a divisor with respect to the Hodge class $\lambda$.  The preimage of $0$ under this map defines a hyperplane in ${\rm{Pic}}_{\mathbb{Q}}\left(X\right)$. Since all effective divisors have positive degree, the NL cone always lies completely on one side of this hyperplane. In the cases where $K_{X}$ not only lies outside of the NL cone but is in fact negative, meaning that it has negative degree with respect to the Hodge class $\lambda$  and so lies on the other side of this hyperplane, we will conclude the additional statement that the orthogonal modular variety $X$ is uniruled and thus has negative Kodaira dimension. This approach is formalized in Proposition \ref{prop: uniruled with branch} below.

In order to describe the canonical class $K_X$ consider the quotient map 
\[\pi\colon \mathcal{D}\rightarrow X=\mathcal{D}\big/\widetilde{\rm{O}}^+\left(\Lambda\right).\]
The map $\pi$ is simply ramified \cite{GHS07}*{Theorem 2.12 and Corollary 2.13} (see also \cite{GHS13}*{Section 6.2}) along the union of hyperplanes $D_\rho$ such that  $\rho\in \Lambda_{\mathbb{Q}}$ is $\widetilde{\rm{O}}^+\left(\Lambda\right)$-reflexive, meaning that $\langle\rho,\rho\rangle<0$ and $\sigma_\rho$ or $-\sigma_\rho$ is in $\widetilde{\rm{O}}^+\left(\Lambda\right)$, where $\sigma_\rho$ is the reflection given by 
\begin{equation}
\label{eq:sec3:ref}
\sigma_\rho:v\mapsto v-2\frac{\langle v,\rho\rangle}{\langle \rho,\rho\rangle}\rho\in {\rm{O}}\left(\Lambda_{\mathbb{Q}}\right).
\end{equation}
The Riemann--Hurwitz formula then yields
    \begin{equation}
\label{sec3:eq:K}K_{X}=n\lambda-\frac{1}{2} {\rm{Br}}(\pi),\end{equation}
    where ${\rm{Br}}(\pi)$ is the reduced branch divisor of $\pi$. 
The vectors $\rho$ contributing to ${\rm{Br}}(\pi)$ are explicitly described in \cite{GHS07}*{Corollary 3.3} as those such that either $\langle \rho,\rho\rangle =-2$ or, letting $D$ be the exponent of the discriminant group $D(\Lambda)$, those such that  $\langle \rho,\rho\rangle =-2D$ and ${\rm div}(\rho)=D\equiv 1\mod 2$ or $\langle \rho,\rho\rangle =-D$ and ${\rm div}(\rho)=D$ or $D/2$. Deducing a formula for  ${\rm{Br}}(\pi)$ then requires understanding the orbit of these $\rho$ under the action of $\widetilde{\rm{O}}^+\left(\Lambda\right)$ and taking the quotient. Since ${\rm{Br}}(\pi)$ is then given in terms of Heegner divisors, this provides a method to explicitly compute $K_X$ in terms of Heegner divisors. 

\begin{remark}\label{rem: contained in branch}
    The fact that the ramification of $\pi$ contains $D_\rho$ such that $\langle \rho,\rho\rangle =-2$ implies for instance that the reduced branch divisor ${\rm{Br}}(\pi)$ always contains $\frac{1}{2}H_{-1,0}$.
\end{remark}

\begin{proposition}\label{prop: uniruled with branch}
Let $\Lambda$ be an even lattice of signature $(2,n)$ with $n\geq3$ splitting off two copies of $U$, let $E_{\frac{n+2}{2},\Lambda}$ be its Eisenstein series, and let ${\rm Br}(\pi)=\sum_{i=1}^r\alpha_{m_i, \mu_i}H_{-m_i, \mu_i}$ with $\alpha_{m_i, \mu_i}\in \mathbb{Q}_{> 0}$ be the reduced branch divisor of the quotient map $\pi\colon \mathcal{D}\rightarrow \mathcal{D}\big/\widetilde{\rm{O}}^+\left(\Lambda\right)$.  If \\\[nc_{0,0}\left(E_{\frac{n+2}{2},\Lambda}\right)+\frac{1}{2}\sum_{i=1}^r \alpha_{m_i, \mu_i} c_{m_i, \mu_i}\left(E_{\frac{n+2}{2},\Lambda}\right)<0,\] then the orthogonal modular variety $X=\mathcal{D}\big/\widetilde{{\rm{O}}}^+(\Lambda)$ is uniruled. 
\end{proposition}

\begin{proof}
Consider the canonical map 
\[\varepsilon\colon \overline{X}^{\rm tor}\rightarrow \overline{X}^{\rm BB}\]
from a toroidal compactification $\overline{X}^{\rm tor}$ of $X$ to its Baily--Borel compactification. 
 Note that toroidal compactifications of locally symmetric manifolds of type ${\rm{O}}(2,n)$ are normal and one can always choose the toroidal data in such a way that the corresponding cones are simplicial, ensuring that they have at worst finite quotient singularities, see \cite{AMRT10}. Since the Hodge class $\lambda$ is ample on $\overline{X}^{\rm BB}$ and $\overline{X}^{\rm BB}\backslash X$ is one-dimensional we can choose a representative for the nef curve class $(\varepsilon^* \lambda)^{n-1}$ which does not meet the boundary divisor $\delta$ of the toroidal compactification  $\overline{X}^{\rm tor}$. 

 By \eqref{sec3:eq:K}, we have 
 \[K_{\overline{X}^{\rm tor}}=n\lambda-\frac{1}{2}{\rm Br}(\pi)-b\delta,\]
 where the value of $b\in \mathbb{Q}$ depends on the choice of the toroidal compactification and the ramification at the boundary. 
Hence if $\eta\colon Y\rightarrow \overline{X}^{\rm tor}$ is a desingularization, then the projection formula and the fact that $\delta.(\varepsilon^* \lambda)^{n-1}=0$ implies 
 \begin{equation}\label{eq: canonical on smoothing}K_Y.\eta^*(\varepsilon^* \lambda)^{n-1}=\left(n\lambda-\frac{1}{2}{\rm Br}(\pi)\right). (\varepsilon^* \lambda)^{n-1}.\end{equation}
 For any divisor  $D$ on $X$, if $\overline{D}$ is the closure of $D$ in $\overline{X}^{\rm tor}$,  since $(\varepsilon^* \lambda)^{n-1}$ does not intersect the boundary of $\overline{X}^{\rm tor}$, then the intersection  $\overline{D}.(\varepsilon^* \lambda)^{n-1}$ is given by the degree of the closure of $D$ in $\overline{X}^{\rm BB}$ with respect to $\lambda$. So consider the map given by the Baily-Borel degree
 \[{\rm deg}\colon {\rm{Pic}}_{\mathbb{Q}}\left(X\right)\longrightarrow \mathbb{Q}.\]
 Theorem \ref{sec2:thm:Pic=Mod} then yields
 \begin{equation}
\label{sec7:eq:Eseries}
\sum_{m,\mu}\left(H_{-m,\mu}\cdot(\varepsilon^*\lambda)^{n-1}\right)q^m\mathfrak{e}_\mu\in {\rm{Mod}}_{\frac{n+2}{2},\Lambda}^\circ.
\end{equation}
Further, by \cite{Kud03}*{Theorem I} (see also \cite{Kud03}*{Corollary 4.12}) this is a multiple of the Eisenstein series $E_{\frac{n+2}{2},\Lambda}$ defined in \eqref{sec2:eq:E}. The nefness of $(\varepsilon^*\lambda)^{n-1}$ then implies that \begin{equation}
\label{sec7:eq:BBdeg}
H_{-m,\mu}\cdot (\varepsilon^*\lambda)^{n-1}=-C\cdot c_{m,\mu}(E_{\frac{n+2}{2},\Lambda}) \;\;\;\hbox{and}\;\;\;(\varepsilon^*\lambda)^{n}=C\cdot c_{0,0}(E_{\frac{n+2}{2},\Lambda}),
\end{equation}
where $c_{m,\mu}\in\left({\rm{Mod}}_{\frac{n+2}{2},\Lambda}^0\right)^\vee$ is the $(m,\mu)$-coefficient extraction functional, and $C$ is a positive constant.  Since ${\rm Br}(\pi)=\sum_{i=1}^r\alpha_{m_i, \mu_i}H_{-m_i, \mu_i}$, combining with \eqref{eq: canonical on smoothing} gives

\begin{equation}\label{eq: degree in terms of heegner}
\begin{aligned}
K_Y.\eta^*(\varepsilon^* \lambda)^{n-1}&=
 \left(n\lambda-\frac{1}{2} \sum_{i=1}^r \alpha_{m_i, \mu_i}H_{-m_i,\mu_i}\right)\cdot\left(\varepsilon^*\lambda\right)^{n-1}\\
&=C\left(nc_{0,0}\left(E_{\frac{n+2}{2},\Lambda}\right)+\frac{1}{2}\sum_{i=1}^r \alpha_{m_i, \mu_i} c_{m_i, \mu_i}\left(E_{\frac{n+2}{2},\Lambda}\right)\right).
\end{aligned}
\end{equation}

Hence if $nc_{0,0}\left(E_{\frac{n+2}{2},\Lambda}\right)+\frac{1}{2}\sum_{i=1}^r \alpha_{m_i, \mu_i} c_{m_i, \mu_i}\left(E_{\frac{n+2}{2},\Lambda}\right)<0$, then $K_Y.\eta^*(\varepsilon^* \lambda)^{n-1}<0$. Since $\eta^*(\varepsilon^* \lambda)^{n-1}$ is nef, it follows that $K_Y$ is not pseudo-effective and so $Y$ is birationally covered by rational curves \cite{MM86,BDPP13}. The uniruledness of $Y$ implies that $X$ is uniruled.     
\end{proof}

In practice, it is easier to apply Proposition \ref{prop: general uniruledness} (stated in the introduction), which follows easily from  Proposition \ref{prop: uniruled with branch}, since it avoids having to explicitly compute ${\rm Br}(\pi)$.

\begin{proof}[Proof of Proposition \ref{prop: general uniruledness}] By Remark \ref{rem: contained in branch} we have $\frac{1}{4}H_{-1,0}\leq \frac{1}{2}{\rm Br}(\pi)$. Thus by \eqref{eq: canonical on smoothing}, \eqref{sec7:eq:BBdeg}, and \eqref{eq: degree in terms of heegner} we have
\begin{equation}\label{eq: curve intersection}
nc_{0,0}\left(E_{\frac{n+2}{2},\Lambda}\right)+\frac{1}{2}\sum_{i=1}^r \alpha_{m_i, \mu_i} c_{m_i, \mu_i}\left(E_{\frac{n+2}{2},\Lambda}\right)\le nc_{0,0}\left(E_{\frac{n+2}{2},\Lambda}\right)+\frac{1}{4}c_{1,0}\left(E_{\frac{n+2}{2},\Lambda}\right)<0
\end{equation}
and so the result follows from Proposition \ref{prop: uniruled with branch}.    
\end{proof}

\begin{remark}
    In fact, in all examples we have computed (see for instance Theorems \ref{sec7:them:OG6uniruled} and \ref{sec7:thm:kum_uniruled} below), the term $c_{1,0}\left(E_{\frac{n+2}{2},\Lambda}\right)$ is so much larger than any of the other terms $c_{m_i, \mu_i}\left(E_{\frac{n+2}{2},\Lambda}\right)$ contributing to $K_Y.\eta^*(\varepsilon^* \lambda)^{n-1}$ in \eqref{eq: degree in terms of heegner} that the approximation $\frac{1}{4}H_{-1,0}\leq \frac{1}{2}{\rm Br}(\pi)$ has no effect on the negativity of $K_Y$, meaning that the left hand side of \eqref{eq: curve intersection} is negative precisely when the right hand side is negative. 
\end{remark}

While in the cases of well-studied moduli spaces of K3 surfaces, hyperk\"ahler manifolds, or cubic fourfolds the strategy for uniruledness of Proposition \ref{prop: uniruled with branch} does not yield new results, we highlight below two lesser-studied cases where we do obtain new results.

\subsection{Moduli of ${\rm{OG6}}$-type hyperk\"ahler manifolds}
Let $(X, L)$ be a primitively polarized hyperk\"ahler sixfold where $X$ is deformation equivalent to O'Grady's six-dimensional example \cite{OGr03}. In this case the Beauville--Bogomolov--Fujiki lattice $\left(H^2\left(X,\mathbb{Z}\right),q_X\right)$ is isomorphic \cite{Rap08} to $\Lambda=U^{\oplus 3}\oplus A_1(-1)^{\oplus 2}$. Further, the monodromy group coincides \cite{MR21} with the full group ${\rm{O}}^+(\Lambda)$. If $h=c_1(L)\in \Lambda$ with $(h,h)=2d>0$, then $\gamma={\rm{div}}_{\Lambda}(h)$ can be $1$ or $2$.

We denote by $\Lambda_h$ the orthogonal complement of $h$ in $\Lambda$.
The period domain $\mathcal{M}_{{\rm{OG6}},2d}^\gamma=\mathcal{D}\big/{\rm{Mon}}^2\left(\Lambda,h\right)$ is a
partial compactification of the moduli space parameterizing primitively polarized hyperk\"{a}hler sixfolds of OG6-type with a polarization of degree $2d$ and divisibility $\gamma$.
It is always irreducible \cite{Son23}*{Section 3}. Moreover, when $\gamma=1$ it is non-empty for all $d\geq 1$ and when $\gamma=2$ it is non-empty only for $d\equiv 2,3\mod 4$. Not much is known about the global geometry of the moduli spaces $\mathcal{M}_{{\rm{OG6}},2d}^\gamma$.

In the split case $\gamma=1$, we have $\Lambda_h\cong U^{\oplus 2}\oplus A_1(-1)^{\oplus 2}\oplus A_1(-d)$. When $\gamma=2$, then $\Lambda_h=U^{\oplus 2}\oplus Q_{t}$, where 
\[
Q_{t}=\begin{cases}A_1(-1)\oplus\left(\begin{array}{cc}-2&1\\1&-2t\end{array}\right)&\hbox{when } d=4t-1\\
\left(\begin{array}{ccc}-2&0&1\\0&-2&1\\1&1&-2t\end{array}\right)&\hbox{when } d=4t-2.\end{cases}
\]

We denote by $\delta_1,\delta_2$ the generators of the two copies of $A_1(-1)$ in $\Lambda$, by $\lbrace e,f \rbrace$ and $\lbrace e_1,f_1 \rbrace$ the canonical basis of two orthogonal copies of the hyperbolic plane.

\begin{lemma}
\label{sec7:lemma:monOG6}
The polarized monodromy group ${\rm{Mon}}^2\left(\Lambda,h\right)\subset{\rm{O}}^+\left(\Lambda_h\right)$ is given by $\widetilde{\rm{O}}^+(\Lambda_h)$ if $\gamma=2$ and $d=4t-1$, and an index two extension of $\widetilde{\rm{O}}^+(\Lambda_h)$ otherwise. More precisely if $\gamma=1$, or $\gamma=2$ and $d=4t-2$, then
\[  {\rm{O}}^+(\Lambda,h)= \langle \widetilde{\rm{O}}^+(\Lambda_h),\sigma_{\kappa} \rangle, \]
where $\kappa=\delta_1-\delta_2$. 

\end{lemma}

\begin{proof}
Let $h\in \Lambda$ be an element with $\langle h,h\rangle=2d$ and ${\rm{div}}_\Lambda(h)=\gamma$.
Since $\Lambda$ and $\Lambda_h$ contain two copies of the hyperbolic plane, the map ${\rm{O}}(\Lambda)\longrightarrow {\rm{O}}(D(\Lambda))$ and the respective one for $\Lambda_h$ are surjective. Note that $D(\Lambda)\cong \left( \mathbb{Z}\big/2\mathbb{Z}\right)^{\oplus 2}$, generated by $\frac{\delta_1}{2}$ and $\frac{\delta_2}{2}$. Further, the only non-trivial element in ${\rm O}(D(\Lambda))$ is the one exchanging the two generators. This in particular implies that $\left[{\rm{O}}^+(\Lambda):\widetilde{\rm{O}}^+(\Lambda)\right]=2$.
Since $\widetilde{\rm{O}}^+(\Lambda,h)=\widetilde{\rm{O}}^+(\Lambda_h)$, see for example \cite{ABL24}*{Lemma 3.15}, either ${\rm{O}}^+(\Lambda, h)$ is equal to $\widetilde{\rm{O}}^+(\Lambda_h)$ or it is a double extension.

By Eichler's Criterion we can always assume $h=e+df$ when $\gamma=1$, and $h=2(e+tf)-\delta$ with $\delta\in\{\delta_1,\delta_2,\delta_1+\delta_2\}$ when $\gamma=2$, the first two happening when $d=4t-1$ and the last one when $d=4t-2$. Note that $\delta_1$ and $\delta_2$ are in the same ${\rm{Mon}}^2(\Lambda)$ orbit so it is enough to consider only one of them. Assume $\gamma=1$ or $\gamma=2$ and $\delta=\delta_1+\delta_2$. Since $\delta_1-\delta_2\in \Lambda_h$, the reflection $\sigma_{\delta_1-\delta_2}\in {\rm{O}}^+(\Lambda)$ fixes $h$ and exchanges the two generators of $D(\Lambda)$, so $\sigma_{\delta_1-\delta_2}\in {\rm{O}}^+(\Lambda,h)={\rm{Mon}}^2\left(\Lambda,h\right)$ and $\sigma_{\delta_1-\delta_2}\not\in\widetilde{\rm{O}}^+(\Lambda_h)$. Finally, assume $\delta=\delta_1$, that is, $\gamma=2$ and $d=4t-1$. If $g\in{\rm{O}}^+(\Lambda,h)$, then $g\left(\frac{h}{2}\right)=g(e+tf)-g\left(\frac{\delta_1}{2}\right)=\frac{h}{2}$. In particular $g\left(\frac{\delta_1}{2}\right)\equiv\frac{\delta_1}{2}\mod \Lambda$. This implies $g\in \widetilde{\rm{O}}^+\left(\Lambda,h\right)$ and we conclude using the equality induced by restriction $\widetilde{\rm{O}}^+\left(\Lambda,h\right)=\widetilde{\rm{O}}^+\left(\Lambda_h\right)$. 
\end{proof}

\begin{theorem}
\label{sec7:them:OG6uniruled}
The moduli space $\mathcal{M}_{{\rm{OG6}},2d}^\gamma$ is uniruled in the following cases
\begin{itemize}
\item[(i)] when $\gamma=1$ for $d\leq12$,
\item[(ii)] when $\gamma=2$ for $t\leq10$ and $t=12$ with $d=4t-1$,
\item[(iii)] when $\gamma=2$ for $t\leq9$ and $t=11, 13$ with $d=4t-2$.
\end{itemize}
\end{theorem}

\begin{proof}
We will show that in the given cases $X=\mathcal{D}\big/\widetilde{\rm{O}}^+\left(\Lambda_h\right)$ is uniruled. Since $\widetilde{\rm{O}}^+\left(\Lambda_h\right)\subset {\rm{Mon}}^2(\Lambda,h)$, there is a dominant map $X\longrightarrow \mathcal{M}_{{\rm{OG6}},2d}^\gamma$ giving us uniruledness for the moduli space $\mathcal{M}_{{\rm{OG6}},2d}^\gamma$. By Proposition \ref{prop: general uniruledness} we just need to verify that $5c_{0,0}\left(E_{\frac{7}{2},\Lambda_h}\right)+\frac{1}{4}c_{-1,0}\left(E_{\frac{7}{2},\Lambda_h}\right)<0$.

There is a concrete formula \cite{BK01} for the coefficients of the Fourier expansion of Eisenstein series. This has been implemented in Sage by the fourth author \cite{weilrep}. We exhibit the highest cases for which we obtain a negative intersection product. When $\gamma=1$ and $d=12$. In this case if we write $E_{\frac{7}{2},\Lambda_{h}}=\sum_{\mu\in D(\Lambda_{h})}E_\mu(q)e_\mu$, then 
\[
E_0(q)=1-\frac{272}{13}q - \frac{1472}{13}q^2 - \frac{3390}{13}q^3 - \frac{8204}{13}q^4 + {\rm{O}}(q^5).
\]
Hence we have 
\[
5c_{0,0}\left(E_{\frac{7}{2},\Lambda_h}\right)+\frac{1}{4}c_{-1,0}\left(E_{\frac{7}{2},\Lambda_h}\right)=5-\frac{272}{52}<0,
\]
which gives the desired uniruledness of $X$ and hence of $\mathcal{M}_{{\rm{OG6}},2d}^\gamma$. Similarly, when $d=4t-1$ and $\gamma=2$, the highest degree for which we obtain uniruledness is $t=12$. In this case the $E_0$ summand of $E_{\frac{7}{2},\Lambda_h}$ is
\[
E_0(q)=1-\frac{1052352}{51911}q - \frac{5438160}{51911}q^2 - \frac{15409296}{51911}q^3 - \frac{907200}{1403}q^4 + {\rm{O}}(q^5)
\]
and so $5c_{0,0}\left(E_{\frac{7}{2},\Lambda_h}\right)+\frac{1}{4}c_{-1,0}\left(E_{\frac{7}{2},\Lambda_h}\right)=5-\frac{1052352}{4\cdot51911}<0,$ as desired. In the case  $\gamma=2$ and $d=4t-2$ with $t=13$ the $E_0$ component of the Eisenstein series of weight $\frac{7}{2}$ corresponding to $\Lambda_h$ is 
\[
E_0(q)=1 - \frac{108}{5}q - \frac{620}{7}q^2 +{\rm{O}}(q^3),
\]
which leads to the same result. The lower-degree cases are done in the same way.
\end{proof}

\subsection{Moduli of ${\rm{Kum}}_n$-type hyperk\"ahler manifolds}
 Let $(X,L)$ be a primitively polarized hyperk\"{a}hler $2n$-fold where $X$ is deformation equivalent to a fiber of the addition map $A^{[n+1]}\longrightarrow A$ on an abelian surface. In this case the $(H^2(X,\mathbb{Z}),q_X)$ is isomorphic to $\Lambda=U^{\oplus 3}\oplus A_1(-(n+1))$ and the monodromy group \cite{Mon16}, \cite{Mar23}*{Theorem 1.4} is:
\begin{equation}
\label{sec7:eq:MonKum}
{\rm{Mon}}^2\left(\Lambda\right)=\left\{\left.g\in\widehat{\rm{O}}^+\left(\Lambda\right)\right| \chi(g)\cdot\det(g)=1\right\},
\end{equation}
where 
$\widehat{\rm{O}}^+\left(\Lambda\right)$ is the group of orientation-preserving isomorphisms of $\Lambda$ which act as $\pm {\rm Id}$ on $D(\Lambda)$ and $\chi:\widehat{\rm{O}}^+\left(\Lambda\right)\longrightarrow\{\pm1\}$ is the character defined by the action of $\widehat{\rm{O}}^+\left(\Lambda\right)$ on $D(\Lambda)$. 
Let $h=c_1(L)\in \Lambda$, with $\langle 
h,h \rangle=2d$ and divisibility $\gamma$. Since $\widetilde{\rm{SO}}^+\left(\Lambda\right)\subset{\rm{Mon}}^2\left(\Lambda\right)$, up to monodromy one can always assume $h=\gamma(e+tf)-a\delta$ for appropriate $t$ and $a$, where $\delta$ is the generator of $A_1(-(n+1))$.

For $\gamma=1,2$, the lattice $\Lambda_h$ is in the form $U^{\oplus 2}\oplus Q_d$ (resp. $Q_t$ with $d=4t-(n+1)$) where
\[
Q_d=\mathbb{Z}\ell+\mathbb{Z}\delta=\left(\begin{array}{cc}-2d&0\\0&-2(n+1)\end{array}\right)\;\;\;\hbox{and}\;\;\;Q_t=\mathbb{Z} u+ \mathbb{Z}v=\left(\begin{array}{cc}-2t&(n+1)\\(n+1)&-2(n+1)\end{array}\right).
\]

\begin{lemma}\label{sec7:lemma:first_MonKum2_2}
For $\gamma=1,2$, the polarized monodromy group ${\rm{Mon}}^2\left(\Lambda,h\right)\subset {\rm{O}}^+\left(\Lambda_h\right)$ is a double extension of $\widetilde{\rm{SO}}^+(\Lambda_h)$. More precisely,
\[{\rm{Mon}}^2\left(\Lambda,h\right)=\langle \widetilde{\rm{SO}}^+(\Lambda_h),\sigma_\kappa \rangle,\;\;\hbox{ where }\;\;
\kappa=
\begin{cases}
    \delta &\mbox{ if }\gamma=1,\\
    v &\mbox{ if }\gamma=2.\\
\end{cases}
\]
For $\gamma\geq 3$ there is equality ${\rm{Mon}}^2\left(\Lambda,h\right)=\widetilde{\rm{SO}}^+(\Lambda_h)$.
\end{lemma}

\begin{proof}
Observe that, for any $g\in {\rm O}\left(\Lambda,h\right)$, we have $\det(g)=\det(g|_{\Lambda_h})$.
Then the statement is essentially \cite{ABL24}*{Lemma 3.7}. For $\gamma=1,2$
we also need to prove that $\sigma_\kappa\in {\rm{Mon}}^2\left(\Lambda\right)$ via the restriction: since $\det(\sigma_\kappa)=-1$, this is equivalent to proving that, if we see $\kappa$ as an element of $\Lambda$, we have $-\sigma_\kappa\in \widetilde{\rm{O}}^+(\Lambda)$ i.e. $\chi(\sigma_\kappa)=-1$. Since $\kappa=3(\gamma-1)f-\delta$, this can be checked via an explicit computation.
\end{proof}

For $\gamma=1,2$, the period domain $\mathcal{M}^\gamma_{{\rm{Kum}}_n,2d}=\mathcal{D}/{\rm{Mon}}^2\left(\Lambda,h\right)$ is a partial compactification of the moduli space of hyperk\"ahler $2n$-folds of generalized Kummer type with a primitive polarization of degree $2d$ and divisibility $\gamma$. It is always irreducible \cite{Ono22} and never empty for 
$\gamma=1$ (the split case). When $\gamma=2$ it is non-empty only for $d\equiv -(n+1)\pmod{4}$.

\begin{lemma}
\label{sec7:lemma:MonKum2_2}
For $d=1$ and $\gamma=1,2$, one has 
\begin{equation}\label{sec7:eq:monkum}
\left\langle{\rm{Mon}}^2\left(\Lambda,h\right),-{\rm{Id}}\right\rangle=\widehat{\rm{O}}^+\left(\Lambda_h\right),    
\end{equation}
or equivalently ${\rm{PMon}}^2\left(\Lambda,h\right)={\rm{P}}\widehat{\rm{O}}^+\left(\Lambda_h\right)={\rm{P}}\widetilde{\rm{O}}^+\left(\Lambda_h\right)$.
\end{lemma}
\begin{proof}
Note that, under our hypothesis, $-\sigma_\kappa\in \widetilde{\rm{O}}^+\left(\Lambda_h\right)$.
For $\gamma=1$, this holds since $\ell_*=-\ell_*$. For $\gamma=2$, observe that $|D(\Lambda_h)|=d\cdot (n+1)=n+1$, hence $D(\Lambda_h)=\langle \kappa_* \rangle$ since $\kappa=v$ is primitive with divisibility $n+1$; clearly $\sigma_\kappa(\kappa_*)=-\kappa_*$.
Now we prove \eqref{sec7:eq:monkum} under the more general hypothesis that $-\sigma_\kappa\in \widetilde{\rm{O}}^+\left(\Lambda_2\right)$ i.e. $\chi(\sigma_\kappa)=-1$. 

We can write $\widehat{\rm{O}}^+\left(\Lambda_h\right)=
\bigcup_{i,j\in \lbrace-1,+1\rbrace} M_{i,j}$,
where $M_{i,j}$ is the set of isometries $g\in \widehat{\rm{O}}^+\left(\Lambda_2\right)$ such that $(\chi(g),\det(g))=(i,j)$.
Clearly $\widetilde{\rm{SO}}^+(\Lambda_h)=M_{+1,+1}$ and, under our hypothesis, $\sigma_\kappa\cdot \widehat{\rm{SO}}^+(\Lambda_h)=M_{-1,-1}$.
By Lemma \ref{sec7:lemma:first_MonKum2_2} then ${\rm{Mon}}^2\left(\Lambda,h\right)=M_{+1,+1}\cup M_{-1,-1}$. Now $-\Id\in M_{-1,1}$, since $\Lambda_h$ has even rank, hence $-\Id\cdot M_{i,j}=M_{-i,j}$ and \eqref{sec7:eq:monkum} follows.
\end{proof}

\begin{theorem}\label{sec7:thm:kum_uniruled}
The moduli spaces $\mathcal{M}_{{\rm{Kum}}_n,2}^1$ and $\mathcal{M}_{{\rm{Kum}}_n,2}^2$ of hyperk\"{a}hler $2n$-folds of generalized Kummer type with polarization of degree $2$ and divisibility $\gamma=1,2$ are uniruled in the following cases:
\begin{itemize}
\item[(i)] when $\gamma=1$ for $n\leq 15$ and $n=17,20$,
\item[(ii)] when $\gamma=2$ for $t\leq 11$ and $t=13,15,17,19$, where $n=4t-2$.
\end{itemize}
\end{theorem}

\begin{proof}
As in the proof of Theorem \ref{sec7:them:OG6uniruled}, we show that $X=\mathcal{D}\big/\widetilde{\rm{O}}^+(\Lambda_h)$ is uniruled and by Lemma \ref{sec7:lemma:MonKum2_2} we conclude uniruledness for $\mathcal{M}_{{\rm{Kum}}_2,2}^\gamma$. By Proposition \ref{prop: general uniruledness}, we just need to verify that $4c_{0,0}(E_{3,\Lambda_h})+\frac{1}{4}c_{-1,0}(E_{3,\Lambda_h})<0$. 
 Again, we exhibit only one case. If $E_{3,\Lambda_h}=\sum_{\mu\in D(\Lambda_h)}E_\mu(q)e_\mu$, one computes \cite{weilrep}:
\[
E_0(q)=\begin{cases}1 - \frac{4250}{263}q - \frac{12600}{263}q^2 + O(q^3)&\hbox{ if $n=20$ and $\gamma=1$}\\
1 - \frac{130}{7}q - \frac{288}{7}q^2 + O(q^3)&\hbox{ if $n=4t-2$ with $t=19$ and $\gamma=2$}.
\end{cases}
\]
\end{proof}

We remark here (see Lemma \ref{sec7:lemma:MonKum2_2}) that the modular variety 
\[
\mathcal{M}^2_{{\rm{Kum}}_2,2}=\mathcal{D}\big/\widetilde{\rm{O}}^+\left(\Lambda_h\right),
\]
where $\Lambda_h=U^{\oplus 2}\oplus A_2(-1)$,
is known to be rational \cite{WW21}*{Theorem 5.4}. More concretely, there is a finite union of Heegner divisors $\mathcal{H}$, see \cite{WW21}*{Equation 5.8}, such that the algebra of meromorphic modular forms ${\rm{M}}_*^!\left(\widetilde{\rm{O}}^+(\Lambda_h),\mathcal{H}\right)$, that is, meromorphic sections of $\lambda^{\otimes k}$ with $k\in \mathbb{Z}$ and poles supported along $\mathcal{H}$, is finitely generated by forms of positive weight. By work of Looijenga \cite{Loo03} the projective variety $\widehat{X}={\rm{Proj}}\left(\oplus_{k\geq0}{\rm{M}}_k^!\left(\widetilde{\rm{O}}^+, \mathcal{H}\right)\right)$ is a compactification of $\mathcal{D}\big/\widetilde{\rm{O}}^+-\mathcal{H}$ that interpolates between the Baily--Borel and toroidal compactifications. When the generators are relation-free, as it is shown in \cite{WW21} for $\Lambda_h=U^{\oplus 2}\oplus A_2(-1)$, the resulting ring is a polynomial algebra with generators of mixed weights. In this case $\widehat{X}$ is a weighted projective space, in particular rational. The same holds for some of the first OG6 cases. Indeed if $\Lambda_h=U^{\oplus 2}\oplus A_1(-1)^{\oplus 3}$ or $\Lambda_h=U^{\oplus 2}\oplus A_1(-1)^{\oplus 1}\oplus A_2(-1)$, then \cite{WW21}*{Theorem 5.4} implies that the resulting modular varieties $\mathcal{D}\big/\widetilde{\rm{O}}^+(\Lambda_h)$ are also rational. We summarize the results relevant for this paper:

\begin{theorem}[Theorem 5.4 in \cite{WW21}]
The moduli space $\mathcal{M}_{{\rm{Kum}}_2,2}^2$ is rational and the moduli spaces $\mathcal{M}_{{\rm{OG6}},6}^2$ and $\mathcal{M}_{{\rm{OG6}},2}^1$ are unirational.
\end{theorem}

\begin{proof}
This is an immediate consequence of \cite{WW21}*{Theorem 5.4} together with Lemmas \ref{sec7:lemma:monOG6} and \ref{sec7:lemma:MonKum2_2}.
\end{proof}

We note that the strategy in Theorem \ref{sec7:thm:kum_uniruled} fails for $\gamma\geq3$. In this case, a nef curve intersecting the canonical class negatively would have to intersect the boundary of a toroidal compactification because the canonical class is always in the interior of the NL cone and in fact is the restriction of an ample class on the Baily--Borel model.

\begin{proposition}\label{prop: higher div}
If $\gamma=3,6$ and the moduli space $\mathcal{M}_{{\rm{Kum}}_2,2d}^\gamma$ is non-empty, then the canonical class of any component $\mathcal{M}$ of the moduli space $\mathcal{M}_{{\rm{Kum}}_2,2d}^\gamma$ 
is given by 
\[
K_{\mathcal{M}}=4\lambda.
\]
In particular, $K_{\mathcal{M}}$ lies in the interior of the NL cone and it has positive intersection with any complete curve not intersecting the boundary of a toroidal compactification.
\end{proposition}

\begin{proof}
By Lemma \ref{sec7:lemma:first_MonKum2_2}, the branch divisor of the modular projection $\pi:\mathcal{D}\longrightarrow \mathcal{D}\big/{\rm{Mon}}^2\left(\Lambda,h\right)$ is trivial, since both $\sigma_\rho$ and $-\sigma_\rho$ have negative determinant on a lattice of even rank, see \eqref{eq:sec3:ref} and its surrounding discussion.
\end{proof}

\begin{remark}
In the case of the rational moduli space  $\left(\mathcal{M}_{{\rm{Kum}}_2,2}^2\right)^\circ$, parameterizing polarized hyperk\"{a}hler fourfolds with polarization of degree $2$ and divisibility $2$, the rational Picard group ${\rm Pic}_\mathbb{Q}\left(\mathcal{M}_{{\rm{Kum}}_2,2}^2\right)$ is one-dimensional, since the space of cusp forms $S_{3,\Lambda_h}$ is trivial (this can be computed using \cite{weilrep}). In particular, in this case we have the equality
\[{\overline{\rm{Eff}}}\left(\mathcal{M}_{{\rm{Kum}}_2,2}^2\right)={\rm{Eff}}^{NL}\left(\mathcal{M}_{{\rm{Kum}}_2,2}^2\right)=\mathbb{Q}_{\geq 0}\lambda,\]
all generated by a single (any) Heegner divisor.
As we detail below, the fact that the rational Picard group is one-dimensional moreover implies the statement of Theorem \ref{thm: quasi-affine} that the moduli space $\left(\mathcal{M}_{{\rm{Kum}}_2,2}^2\right)^\circ$ is quasi-affine, meaning that $\left(\mathcal{M}_{{\rm{Kum}}_2,2}^2\right)^\circ$ is an open subset of an affine variety. The statement of Corollary \ref{cor: isotrivial} that any family of polarized hyperk\"ahler fourfolds in $\left(\mathcal{M}_{{\rm{Kum}}_2,2}^2\right)^\circ$ that lies over a projective base must be isotrivial then follows immediately.
    
\end{remark}

\begin{proof}[Proof of Theorem \ref{thm: quasi-affine}]
Recall that via the period map \cite{Ver13} and 
 Lemma \ref{sec7:lemma:first_MonKum2_2} we have an open embedding
\begin{equation}
\label{sec7:eq:torelli_Kum}
\left(\mathcal{M}_{{\rm{Kum}}_2,2}^2\right)^\circ\longrightarrow \mathcal{D}\big/{\rm{Mon}}^2\left(\Lambda,h\right)=\mathcal{D}\big/\widetilde{\rm{O}}^+\left(\Lambda_h\right)=\mathcal{M}_{{\rm{Kum}}_2,2}^2.
\end{equation}
Since ${\rm Pic}_\mathbb{Q}\left(\mathcal{M}_{{\rm{Kum}}_2,2}^2\right)$ is one-dimensional, it is thus enough to show that  
 the complement of $\left(\mathcal{M}_{{\rm{Kum}}_2,2}^2\right)^\circ$ in $\mathcal{M}_{{\rm{Kum}}_2,2}^2$ contains a primitive Heegner divisor $P_\rho$ since then $P_\rho$ is a positive rational multiple of the Hodge class $\lambda$ and so is ample, meaning that $\left(\mathcal{M}_{{\rm{Kum}}_2,2}^2\right)^\circ$ is an open subset of the complement of a hyperplane in the Baily-Borel compactification $\mathcal{D}\big/\widetilde{\rm{O}}^+\left(\Lambda_h\right)^{\rm BB}\subset \mathbb{P}^N$. This exactly means that $\left(\mathcal{M}_{{\rm{Kum}}_2,2}^2\right)^\circ$ is quasi-affine.

Recall that if $(X,H)$ is a polarized hyperk\"{a}hler fourfold of ${\rm{Kum}}_2$-type, then 
\[
\left(H^2\left(X,\mathbb{Z}\right),q_X\right)\cong \Lambda
\]
with $\Lambda=U^{\oplus 3}\oplus A_1(-3)$. We call  the generator of the last factor $\delta$ and let $h=c_1(H)$. By \cite{Yos16}, see also \cite{MTW2018}*{Page 452}, an ample class $h$ cannot lie in the orthogonal complement in $H^{1,1}(X,\mathbb{R})$ of classes $\rho\in {\rm{NS}}(X)$ whose square is $-6$ and divisibility in $H^2(X,\mathbb{Z})$ is $2,3$ or $6$. In particular, if such a class is orthogonal to $h$, then $D_\rho$ defines a hyperplane in $\mathcal{D}$ and the image of the period map misses the corresponding divisor $P_\rho$. Singling out classes in $H^{1,1}(X,\mathbb{Z})$ whose orthogonal complements give the chamber decomposition of the positive cone $C(X)\subset H^{1,1}(X,\mathbb{R})$ is a general method to describe the complement of the image of the period map, see for instance \cite{DM19}*{Theorem 6.1}. Thus, it is enough to show that there exists an integral class $\rho\in\Lambda_h$ of square $\langle\rho,\rho\rangle=-6$ and divisibility  $2,3,$ or $6$ in $\Lambda$. Since $\widetilde{\rm{SO}}^+(\Lambda)\subset {\rm{Mon}}^2(\Lambda)$, one can assume $h=2(e+f)-\delta$, and taking $\rho=3f-\delta$ one has the desired property. In this case the divisibility in $\Lambda$ is $3$, and the missed primitive Heegner divisor in $\mathcal{M}_{{\rm{Kum}}_2,2}^2$ is $P_\rho=P_{-\frac{1}{3},\frac{v}{3}}$. 
\end{proof}

\begin{remark} We have seen that one can compute the NL cone of  $\mathcal{M}_{{\rm{Kum}}_2,2}^2$ just from the fact that $\dim_\mathbb{Q}{\rm Pic}_\mathbb{Q}\left(\mathcal{M}_{{\rm{Kum}}_2,2}^2\right)=1$. 
The result of Lemma \ref{sec7:lemma:first_MonKum2_2} a priori means that one can also compute the NL cone of the moduli space
    $\mathcal{M}_{{\rm{Kum}}_2,2}^1$ using the method of Section \ref{sec: NL cone comp} for the quotient $\mathcal{D}\big/\widetilde{\rm{O}}^+\left(\Lambda_h\right)$. In this case, however, the bound obtained in Theorem \ref{thm: NL cone generators} is too large for this to be computationally feasible. A computation of the cone generated by all $P_{\Delta, \delta}$ with $\Delta\le 10000$ yields the conjectural description
\[
{\rm{Eff}}^{NL}\left(\mathcal{M}_{{\rm{Kum}}_2,2}^1\right)=\left\langle P_{-\frac{1}{12},\delta_*}, P_{-\frac{1}{4},\ell_*}\right\rangle_{\mathbb{Q}_{\geq0}}.
\]
\end{remark}

\section{Declarations}
\subsection{Funding} Ignacio Barros was supported by the Research Foundation – Flanders (FWO) within the framework of the Odysseus program project number G0D9323N and by the Deutsche Forschungsgemeinschaft (DFG, German Research Foundation) -- SFB-TRR 358/1 2023 -- 491392403. 
Pietro Beri was supported by the ERC Synergy Grant ERC-2020-SyG-854361-HyperK. 
Laure Flapan was supported by NSF grant DMS-2200800. Brandon Williams did not receive funding for the preparation of this manuscript. 

\subsection{Financial Interests} The authors have no relevant financial or non-financial interests to disclose.

\subsection{Conflict of Interest} On behalf of all authors, the corresponding author states that there is no conflict of interest.

\bibliography{Bibliography}
\bibliographystyle{alpha}


\addresseshere

\section{Appendix}

 For $X=\mathcal{D}\big/\widetilde{\rm{O}}^+\left(\Lambda_h\right)$, we record here details about ${\rm Eff}^{NL}(X)$ including the set of minimal generating rays, the number of these rays, and the dimension of the $\mathbb{Q}$-vector space ${\rm{Pic}}_\mathbb{Q}\left(X\right)$.


\begin{longtable}{|c|l|c|c|}
\caption{The NL cone of $\mathcal{F}_{2d}$.}
\label{sec6:table:NLK3}\\
\hline
$d$ & minimal generating rays of ${\rm{Eff}}^{NL}\left(\mathcal{F}_{2d}\right)$& $\#$ rays& dim \\
\hline
$1$&$P_{-1, 0}, P_{-\frac{1}{4}, \ell_*}$&$2$&$2$\\

\hline
$2$&$P_{-1, 0}, P_{-\frac{1}{8}, \ell_*}, P_{-\frac{1}{2}, 2\ell_*}$&$3$&$3$\\

\hline
$3$&$P_{-1, 0}, P_{-\frac{1}{12}, \ell_*}, P_{-\frac{1}{3}, 2\ell_*}, P_{-\frac{3}{4}, 3\ell_*}$&$4$&$4$\\

\hline
$4$&$P_{-1, 0}, P_{-\frac{1}{16}, \ell_*}, P_{-\frac{1}{4}, 2\ell_*}, P_{-\frac{9}{16}, 3\ell_*}, P_{-1, 4\ell_*}$&$5$&$4$\\

\hline
$5$&$P_{-1, 0}, P_{-\frac{1}{20}, \ell_*}, P_{-\frac{1}{5}, 2\ell_*}, P_{-\frac{9}{20}, 3\ell_*}, P_{-\frac{4}{5}, 4\ell_*}, P_{-\frac{1}{4}, 5\ell_*}$&$6$&$6$\\

\hline
$6$&$P_{-1, 0}, P_{-\frac{1}{24}, \ell_*}, P_{-\frac{1}{6}, 2\ell_*}, P_{-\frac{3}{8}, 3\ell_*}, P_{-\frac{2}{3}, 4\ell_*}, P_{-\frac{1}{24}, 5\ell_*},P_{-\frac{1}{2}, 6\ell_*}$&$7$&$7$\\
\hline
$7$&$P_{-1, 0}, P_{-\frac{1}{28}, \ell_*}, P_{-\frac{1}{7}, 2\ell_*}, P_{-\frac{9}{28}, 3\ell_*}, P_{-\frac{4}{7}, 4\ell_*}, P_{-\frac{25}{28}, 5\ell_*},P_{-\frac{2}{7}, 6\ell_*}, P_{-\frac{3}{4}, 7\ell_*}$&$8$&$7$\\
\hline
$8$&$P_{-1, 0}, P_{-\frac{1}{32}, \ell_*}, P_{-\frac{33}{32}, \ell_*}, P_{-\frac{1}{8}, 2\ell_*}, P_{-\frac{9}{32}, 3\ell_*}, P_{-\frac{1}{2}, 4\ell_*},P_{-\frac{25}{32}, 5\ell_*}, P_{-\frac{1}{8}, 6\ell_*}, $&$10$&$8$\\
&$P_{-\frac{17}{32}, 7\ell_*}, P_{-1, 8\ell_*}$&&\\
\hline
$9$&$P_{-1, 0}, P_{-\frac{1}{36}, \ell_*}, P_{-\frac{37}{36}, \ell_*}, P_{-\frac{1}{9}, 2\ell_*}, P_{-\frac{10}{9}, 2\ell_*}, P_{-\frac{1}{4}, 3\ell_*},P_{-\frac{4}{9}, 4\ell_*}, P_{-\frac{25}{36}, 5\ell_*},$&$13$&$9$\\
&$ P_{-1, 6\ell_*}, P_{-\frac{13}{36}, 7\ell_*}, P_{-\frac{7}{9}, 8\ell_*},P_{-\frac{1}{4}, 9\ell_*},P_{-\frac{5}{4}, 9\ell_*}$&&\\
\hline
$10$&$P_{-1, 0}, P_{-\frac{1}{40}, \ell_*}, P_{-\frac{1}{10}, 2\ell_*}, P_{-\frac{9}{40}, 3\ell_*}, P_{-\frac{2}{5}, 4\ell_*}, P_{-\frac{5}{8}, 5\ell_*},P_{-\frac{9}{10}, 6\ell_*}, P_{-\frac{9}{40}, 7\ell_*},$&$11$&$10$\\
&$ P_{-\frac{3}{5}, 8\ell_*}, P_{-\frac{1}{40}, 9\ell_*}, P_{-\frac{1}{2}, 10\ell_*}$&&\\
\hline
$11$&$P_{-1, 0}, P_{-\frac{1}{44}, \ell_*}, P_{-\frac{45}{44}, \ell_*}, P_{-\frac{1}{11}, 2\ell_*}, P_{-\frac{12}{11}, 2\ell_*}, P_{-\frac{9}{44}, 3\ell_*},P_{-\frac{53}{44}, 3\ell_*}, P_{-\frac{4}{11}, 4\ell_*},$&$16$&$11$\\
&$ P_{-\frac{25}{44}, 5\ell_*}, P_{-\frac{9}{11}, 6\ell_*}, P_{-\frac{5}{44}, 7\ell_*},P_{-\frac{49}{44}, 7\ell_*}, P_{-\frac{5}{11}, 8\ell_*}, P_{-\frac{37}{44}, 9\ell_*}, P_{-\frac{3}{11}, 10\ell_*},$&&\\
&$ P_{-\frac{3}{4}, 11\ell_*}$&&\\
\hline
$12$&$P_{-1, 0}, P_{-\frac{1}{48}, \ell_*}, P_{-\frac{49}{48}, \ell_*}, P_{-\frac{1}{12}, 2\ell_*}, P_{-\frac{3}{16}, 3\ell_*}, P_{-\frac{1}{3}, 4\ell_*},P_{-\frac{25}{48}, 5\ell_*}, P_{-\frac{3}{4}, 6\ell_*},$&$15$&$12$\\
&$ P_{-\frac{1}{48}, 7\ell_*}, P_{-\frac{49}{48}, 7\ell_*}, P_{-\frac{1}{3}, 8\ell_*},P_{-\frac{11}{16}, 9\ell_*}, P_{-\frac{1}{12}, 10\ell_*}, P_{-\frac{25}{48}, 11\ell_*}, P_{-1, 12\ell_*}$&&\\
\hline
$13$&$P_{-1, 0}, P_{-\frac{1}{52}, \ell_*}, P_{-\frac{53}{52}, \ell_*}, P_{-\frac{1}{13}, 2\ell_*}, P_{-\frac{9}{52}, 3\ell_*}, P_{-\frac{4}{13}, 4\ell_*},P_{-\frac{25}{52}, 5\ell_*}, P_{-\frac{9}{13}, 6\ell_*},$&$16$&$12$\\
&$ P_{-\frac{49}{52}, 7\ell_*}, P_{-\frac{3}{13}, 8\ell_*}, P_{-\frac{29}{52}, 9\ell_*},P_{-\frac{12}{13}, 10\ell_*}, P_{-\frac{17}{52}, 11\ell_*}, P_{-\frac{10}{13}, 12\ell_*}, P_{-\frac{1}{4}, 13\ell_*},$&&\\
&$ P_{-\frac{5}{4}, 13\ell_*}$&&\\
\hline
$14$&$P_{-1, 0}, P_{-\frac{1}{56}, \ell_*}, P_{-\frac{57}{56}, \ell_*}, P_{-\frac{1}{14}, 2\ell_*}, P_{-\frac{9}{56}, 3\ell_*}, P_{-\frac{2}{7}, 4\ell_*},P_{-\frac{25}{56}, 5\ell_*}, P_{-\frac{9}{14}, 6\ell_*},$&$18$&$14$\\
&$ P_{-\frac{7}{8}, 7\ell_*}, P_{-\frac{1}{7}, 8\ell_*}, P_{-\frac{25}{56}, 9\ell_*},P_{-\frac{11}{14}, 10\ell_*}, P_{-\frac{9}{56}, 11\ell_*}, P_{-\frac{4}{7}, 12\ell_*}, P_{-\frac{1}{56}, 13\ell_*},$&&\\
&$P_{-\frac{57}{56}, 13\ell_*}, P_{-\frac{1}{2}, 14\ell_*}, P_{-\frac{3}{2}, 14\ell_*}$&&\\
\hline
$15$&$P_{-1, 0}, P_{-\frac{1}{60}, \ell_*}, P_{-\frac{61}{60}, \ell_*}, P_{-\frac{1}{15}, 2\ell_*}, P_{-\frac{16}{15}, 2\ell_*}, P_{-\frac{3}{20}, 3\ell_*},P_{-\frac{4}{15}, 4\ell_*}, P_{-\frac{5}{12}, 5\ell_*},$&$20$&$15$\\
&$ P_{-\frac{3}{5}, 6\ell_*}, P_{-\frac{49}{60}, 7\ell_*}, P_{-\frac{1}{15}, 8\ell_*},P_{-\frac{16}{15}, 8\ell_*}, P_{-\frac{7}{20}, 9\ell_*}, P_{-\frac{2}{3}, 10\ell_*}, P_{-\frac{1}{60}, 11\ell_*},$&&\\
&$ P_{-\frac{61}{60}, 11\ell_*},P_{-\frac{2}{5}, 12\ell_*}, P_{-\frac{49}{60}, 13\ell_*}, P_{-\frac{4}{15}, 14\ell_*}, P_{-\frac{3}{4}, 15\ell_*}$&&\\
\hline
$16$&$P_{-1, 0}, P_{-\frac{1}{64}, \ell_*}, P_{-\frac{65}{64}, \ell_*}, P_{-\frac{1}{16}, 2\ell_*}, P_{-\frac{17}{16}, 2\ell_*}, P_{-\frac{9}{64}, 3\ell_*},P_{-\frac{1}{4}, 4\ell_*}, P_{-\frac{25}{64}, 5\ell_*},$&$20$&$14$\\
&$ P_{-\frac{9}{16}, 6\ell_*}, P_{-\frac{49}{64}, 7\ell_*}, P_{-1, 8\ell_*},P_{-\frac{17}{64}, 9\ell_*}, P_{-\frac{9}{16}, 10\ell_*}, P_{-\frac{57}{64}, 11\ell_*}, P_{-\frac{1}{4}, 12\ell_*},$&&\\
&$ P_{-\frac{41}{64}, 13\ell_*},P_{-\frac{1}{16}, 14\ell_*}, P_{-\frac{17}{16}, 14\ell_*}, P_{-\frac{33}{64}, 15\ell_*}, P_{-1, 16\ell_*}$&&\\
\hline
$17$&$P_{-1, 0}, P_{-\frac{1}{68}, \ell_*}, P_{-\frac{69}{68}, \ell_*}, P_{-\frac{1}{17}, 2\ell_*}, P_{-\frac{18}{17}, 2\ell_*}, P_{-\frac{9}{68}, 3\ell_*},P_{-\frac{77}{68}, 3\ell_*}, P_{-\frac{4}{17}, 4\ell_*},$&$23$&$16$\\
&$ P_{-\frac{25}{68}, 5\ell_*}, P_{-\frac{9}{17}, 6\ell_*}, P_{-\frac{49}{68}, 7\ell_*},P_{-\frac{16}{17}, 8\ell_*}, P_{-\frac{13}{68}, 9\ell_*}, P_{-\frac{8}{17}, 10\ell_*}, P_{-\frac{53}{68}, 11\ell_*},$&&\\
&$ P_{-\frac{2}{17}, 12\ell_*},P_{-\frac{19}{17}, 12\ell_*}, P_{-\frac{33}{68}, 13\ell_*}, P_{-\frac{15}{17}, 14\ell_*}, P_{-\frac{21}{68}, 15\ell_*}, P_{-\frac{13}{17}, 16\ell_*},P_{-\frac{1}{4}, 17\ell_*},$&&\\
&$ P_{-\frac{5}{4}, 17\ell_*}$&&\\
\hline
$18$&$P_{-1, 0}, P_{-\frac{1}{72}, \ell_*}, P_{-\frac{73}{72}, \ell_*}, P_{-\frac{1}{18}, 2\ell_*}, P_{-\frac{19}{18}, 2\ell_*}, P_{-\frac{1}{8}, 3\ell_*},P_{-\frac{2}{9}, 4\ell_*}, P_{-\frac{11}{9}, 4\ell_*},$&$25$&$17$\\
&$ P_{-\frac{25}{72}, 5\ell_*}, P_{-\frac{1}{2}, 6\ell_*}, P_{-\frac{49}{72}, 7\ell_*},P_{-\frac{8}{9}, 8\ell_*}, P_{-\frac{1}{8}, 9\ell_*}, P_{-\frac{9}{8}, 9\ell_*}, P_{-\frac{7}{18}, 10\ell_*},$&&\\
&$ P_{-\frac{49}{72}, 11\ell_*},P_{-1, 12\ell_*}, P_{-\frac{25}{72}, 13\ell_*}, P_{-\frac{13}{18}, 14\ell_*}, P_{-\frac{1}{8}, 15\ell_*}, P_{-\frac{5}{9}, 16\ell_*},P_{-\frac{1}{72}, 17\ell_*},$&&\\
&$ P_{-\frac{73}{72}, 17\ell_*}, P_{-\frac{1}{2}, 18\ell_*}, P_{-\frac{3}{2}, 18\ell_*}$&&\\
\hline
$19$&$P_{-1, 0}, P_{-\frac{1}{76}, \ell_*}, P_{-\frac{77}{76}, \ell_*}, P_{-\frac{1}{19}, 2\ell_*}, P_{-\frac{20}{19}, 2\ell_*}, P_{-\frac{9}{76}, 3\ell_*},P_{-\frac{85}{76}, 3\ell_*}, P_{-\frac{4}{19}, 4\ell_*},$&$27$&$17$\\
&$ P_{-\frac{23}{19}, 4\ell_*}, P_{-\frac{25}{76}, 5\ell_*}, P_{-\frac{9}{19}, 6\ell_*},P_{-\frac{49}{76}, 7\ell_*}, P_{-\frac{16}{19}, 8\ell_*}, P_{-\frac{5}{76}, 9\ell_*}, P_{-\frac{81}{76}, 9\ell_*},$&&\\
&$ P_{-\frac{6}{19}, 10\ell_*},P_{-\frac{45}{76}, 11\ell_*}, P_{-\frac{17}{19}, 12\ell_*}, P_{-\frac{17}{76}, 13\ell_*}, P_{-\frac{93}{76}, 13\ell_*}, P_{-\frac{11}{19}, 14\ell_*},P_{-\frac{73}{76}, 15\ell_*},$&&\\
&$ P_{-\frac{7}{19}, 16\ell_*}, P_{-\frac{61}{76}, 17\ell_*}, P_{-\frac{5}{19}, 18\ell_*},P_{-\frac{24}{19}, 18\ell_*}, P_{-\frac{3}{4}, 19\ell_*}$&&\\
\hline
$20$&$P_{-1, 0}, P_{-2, 0}, P_{-\frac{1}{80}, \ell_*}, P_{-\frac{81}{80}, \ell_*}, P_{-\frac{1}{20}, 2\ell_*}, P_{-\frac{21}{20}, 2\ell_*},P_{-\frac{9}{80}, 3\ell_*}, P_{-\frac{89}{80}, 3\ell_*},$&$28$&$19$\\
&$ P_{-\frac{1}{5}, 4\ell_*}, P_{-\frac{5}{16}, 5\ell_*}, P_{-\frac{9}{20}, 6\ell_*}, P_{-\frac{49}{80}, 7\ell_*}, P_{-\frac{4}{5}, 8\ell_*}, P_{-\frac{1}{80}, 9\ell_*}, P_{-\frac{81}{80}, 9\ell_*},$&&\\
&$ P_{-\frac{1}{4}, 10\ell_*},P_{-\frac{41}{80}, 11\ell_*}, P_{-\frac{4}{5}, 12\ell_*}, P_{-\frac{9}{80}, 13\ell_*}, P_{-\frac{89}{80}, 13\ell_*}, P_{-\frac{9}{20}, 14\ell_*},P_{-\frac{13}{16}, 15\ell_*},$&&\\
&$ P_{-\frac{1}{5}, 16\ell_*}, P_{-\frac{49}{80}, 17\ell_*}, P_{-\frac{1}{20}, 18\ell_*}, P_{-\frac{21}{20}, 18\ell_*},P_{-\frac{41}{80}, 19\ell_*}, P_{-1, 20\ell_*}$&&\\
&&&\\
\hline
\end{longtable}

\begin{table}[h]
\caption{The NL cone of $\mathcal{M}_{{\rm{K3}}^{[2]},2d}^1$ with $d\leq 5$}
\label{sec6:table:NLK32_sp}
\begin{tabular}{|c|l|c|c|}
\hline
$d$&minimal generating rays of ${\rm{Eff}}^{NL}\left(\mathcal{M}_{{\rm{K3}}^{[2]},2d}^1\right)$&$\#$ rays & dim \\
\hline
$1$&$P_{-1,0}, P_{-\frac{1}{4},\ell_*}, P_{-\frac{1}{4},\delta_*}, P_{-\frac{1}{2},\ell_*+\delta_*}$&$4$&$4$\\
\hline
$2$&$P_{-1,0}, P_{-\frac{1}{8},\ell_*}, P_{-\frac{9}{8},\ell_*}, P_{-\frac{1}{4},\delta_*},P_{-\frac{5}{4},\delta_*}, P_{-\frac{3}{8}, \ell_*+\delta_*},P_{-\frac{1}{2},2\ell_*},P_{-\frac{3}{4},2\ell_*+\delta_*}
$&$8$& $6$\\
\hline
$3$&$P_{-1,0},P_{-\frac{1}{12},\ell_*}, P_{-\frac{13}{12},\ell_*}, P_{-\frac{1}{4},\delta_*}, P_{-\frac{1}{3},\ell_*+\delta_*},P_{-\frac{1}{3},2\ell_*},P_{-\frac{7}{12},2\ell_*+\delta_*},P_{-\frac{3}{4},3\ell_*},$&$9$&$7$\\
&$P_{-1,3\ell_*+\delta_*}$&&\\
\hline
$4$&$P_{-1,0}, P_{-\frac{1}{16},\ell_*}, P_{-\frac{17}{16},\ell_*},  P_{-\frac{1}{4},\delta_*}, P_{-\frac{5}{16},\ell_*+\delta_*},P_{-\frac{1}{4},2\ell_*},P_{-\frac{5}{4},2\ell_*},P_{-\frac{1}{2},2\ell_*+\delta_*},$&$12$&$9$\\
&$ P_{-\frac{9}{16},3\ell_*}, P_{-\frac{13}{16},3\ell_*+\delta_*}, P_{-1,4\ell_*}, P_{-\frac{1}{4},4\ell_*+\delta_*} $&&\\
\hline
$5$&$ P_{-1,0}, P_{-\frac{1}{20},\ell_*},P_{-\frac{21}{20},\ell_*}, P_{-\frac{1}{4},\delta_*}, P_{-\frac{3}{10},\ell_*+\delta_*}, P_{-\frac{1}{5},2\ell_*},P_{-\frac{6}{5},2\ell_*}, P_{-\frac{9}{20},2\ell_*+\delta_*},$&$16$&$12$\\
&$ P_{-\frac{9}{20},3\ell_*}, P_{-\frac{7}{10},3\ell_*+\delta_*},P_{-\frac{4}{5},4\ell_*},P_{-\frac{1}{20},4\ell_*+\delta_*},P_{-\frac{21}{20},4\ell_*+\delta_*}, , P_{-\frac{1}{4},5\ell_*}, $&&\\
&$P_{-\frac{5}{4},5\ell_*},P_{-\frac{1}{2},5\ell_*+\delta_*}$&&\\
&&&\\
\hline
\end{tabular}
\end{table}

\begin{table}[h]
\caption{The NL cone of $\mathcal{M}_{{\rm{K3}}^{[2]},8t-2}^2$ with $t\leq 5$}
\label{sec6:table:NLK32_nsp}
\begin{tabular}{|c|l|c|c|}
\hline
$t$&minimal generating rays of ${\rm{Eff}}^{NL}\left(\mathcal{M}_{{\rm{K3}}^{[2]},8t-2}^2\right)$& $\#$ rays &dim\\
\hline
$1$&$P_{-1,0},P_{-\frac{1}{3},\frac{2u_+v}{3}}$&$2$&$2$  \\
\hline
$2$&$P_{-1,0}, P_{-\frac{1}{7},\frac{2u+v}{7}},P_{-\frac{4}{7}, \frac{4u_+2v}{7}}, P_{-\frac{2}{7}, \frac{6u+3v}{7}} $&$4$&$4$  \\
\hline
$3$&$P_{-1,0}, P_{-\frac{3}{11},\frac{u+6v}{11}}, P_{-\frac{1}{11},\frac{2u_+v}{11}}, P_{-\frac{12}{11},\frac{2u+v}{11}}, P_{-\frac{4}{11},\frac{4u+2v}{11}},P_{-\frac{9}{11},\frac{6u+3v}{11}}, $&$7$&$6$ \\
&$P_{-\frac{5}{11},\frac{8u+4v}{11}}$&&\\
\hline
$4$&$P_{-1,0}, P_{-\frac{2}{3},\frac{u+2v}{3}}, P_{-\frac{4}{15},\frac{u+8v}{15}}, P_{-\frac{1}{15},\frac{2u_+v}{15}}, P_{-\frac{16}{15},\frac{2u+v}{15}}, P_{-\frac{2}{5},\frac{3u+2v}{5}},P_{-\frac{3}{5},\frac{3u+4v}{5}}, $&$10$&$8$  \\
&$P_{-\frac{4}{15},\frac{4u+2v}{15}},P_{-\frac{1}{15},\frac{7u+11v}{15}}, P_{-\frac{16}{15},\frac{7u+11v}{15}}$&&\\
\hline
$5$&$P_{-1,0}, P_{-\frac{5}{19},\frac{u+10v}{19}},P_{-\frac{1}{19},\frac{2u+v}{19}},P_{-\frac{20}{19},\frac{2u+v}{19}},P_{-\frac{7}{19},\frac{3u+11v}{19}}, P_{-\frac{4}{19},\frac{4u+2v}{19}},
$ & $12$& $9$  \\
&$ P_{-\frac{23}{19},\frac{4u+2v}{19}}, P_{-\frac{11}{19},\frac{5u+12v}{19}},P_{-\frac{9}{19},\frac{6u+3v}{19}},P_{-\frac{17}{19},\frac{7u+13v}{19}}P_{-\frac{16}{19},\frac{8u+4v}{19}}, P_{-\frac{6}{19},\frac{9u+14v}{19}}$ &&\\
&&&\\
\hline
\end{tabular}
\end{table}

\end{document}